\documentclass[11pt,twoside]{article}
\usepackage{latexsym,amsmath,amssymb,times,tikz,pgfplots}

\newif\ifdviwin

\renewenvironment{thebibliography}[1]{
  \begin{oldthebibliography}{#1}
    \setlength{\itemsep}{0em}
    \setlength{\parskip}{0em}
}
{
  \end{oldthebibliography}
}

\usepackage{titlesec}
	
\titleformat{\section}{\large\bfseries\center}{\thesection}{1em}{\vspace{.0cm}}
\titleformat{\subsection}[runin]{\bfseries}{\thesubsection}{1em}{}

\usepackage[english]{babel}
\usepackage{indentfirst}
\usepackage[mathscr]{eucal}
\usepackage{amssymb,amsmath,amsfonts}
\usepackage{fancybox,fancyhdr}
\usepackage{graphicx}
\usepackage[utf8]{inputenc}
\usepackage{float}
\usepackage{color}
\usepackage[pdftex]{hyperref}
\hypersetup{colorlinks=true,linkcolor=blue,citecolor=blue}

\newif\ifdviwin

\dviwintrue

 \newtheorem{defi}{Definition}[section]
 \newtheorem{teo}[defi]{Theorem}
 \newtheorem{pro}[defi]{Proposition}
 \newtheorem{cor}[defi]{Corollary}

 \newenvironment{proof}{\rm \trivlist \item[\hskip \labelsep{\it
      Proof}.]}{\nopagebreak \hfill $\Box$ \endtrivlist}

\numberwithin{equation}{section}

\def\r{\mathbb{R}}
\def\R{\mathcal{R}}
\def\gf{\mathcal{G}_f}
\def\gl{\mathcal{G}_l}
\def\gL{\mathcal{G}_L}
\def\nil{\mathrm{Nil}_3}
\def\H{\mathbb{H}}
\def\s{\mathbb{S}}

\begin{document}
\thispagestyle{empty}

\begin{center}

\renewcommand{\thefootnote}{\,}
{\Large \bf Invariant $\lambda$-translators in the Heisenberg group
\footnote{\hspace{-.75cm}
\emph{Mathematics Subject Classification:} \\
\emph{Keywords}: }}\\
\vspace{0.5cm} { Antonio Bueno}\\
\end{center}
\vspace{.5cm}

\begin{abstract}
We study oriented surfaces in the Heisenberg space $\nil$ whose mean curvature $H$ at each point is $H=\langle N,\partial_z\rangle+\lambda$, where $N$ is the unit normal, $\partial_z$ is the vertical Killing vector field and $\lambda\in\r$. These surfaces are known as $\lambda$-translators and generalize, among others, minimal and positive constant mean curvature surfaces, and also translating solitons of the mean curvature flow. The objective in this paper is to classify $\lambda$-translators invariant by the following one-parameter groups of isometries of $\nil$: left-translations, rotations and helicoidal motions.
\end{abstract}

\section{Introduction}

Let $M$ be a surface in a Riemannian 3-manifold $(\mathcal{M},g)$ and $\{\Psi_t\}$ a smooth variation of $M$. We say that $\Psi_t$ evolves by mean curvature flow if
$$
\left(\frac{\partial\Psi_t}{\partial t}(p,t)\right)^\bot=H_t(\Psi_t(p))N_t(\Psi_t(p)),
$$
where $H_t$ and $N_t$ are the mean curvature and unit normal of $\Psi_t(M)$ and $p\in M$. The surface is a translating soliton of the mean curvature flow, or simply a \emph{translator}, if $M$ evolves by mean curvature flow under a one-parameter group of ambient isometries. In such a case, if $X$ is the Killing vector field associated to the group of isometries, then $M$ is a soliton if and only if its mean curvature satisfies $H=g(N,X)$. 

For the particular but important case that $\mathcal{M}=\r^3$ and we take $\Psi_t$ as translations in a fixed direction $\vec{v}$, translators are just surfaces whose mean curvature is given by $H=\langle N,\vec{v}\rangle$. The importance of translators is that they appear in the singularity theory of the MCF as the equation of the limit flow by a blow-up procedure near type II singularities; see \cite{hui,hs}, although a treatment from the PDE viewpoint of the solvability of the Dirichlet problem was addressed in Serrin's seminal paper \cite{ser}. In the last decades, there has been an active and fruitful research regarding this topic. The most simple case is the one-dimensional problem in the plane: the only curve that evolves by the flow of its curvature is the \emph{grim reaper} curve $y=-\log\cos x$. In $\r^3$ we have as cylindrical translators any plane containing the direction $\vec{v}$, the \emph{grim reaper cylinder} \cite{mpshs} (which is the cylindrical surface erected over the grim reaper curve) and a one-parameter family of \emph{tilted} grim reapers varying between the former examples. Regarding rotational examples, we find the bowl soliton and the wing-like examples \cite{css} as the analogs to the plane and the one-parameter family of minimal catenoids, and more generally helicoidal surfaces \cite{ha}. We refer the reader to \cite{himw,sx} and references therein for an outline of the development of this theory. 

The generality in the definition of a translator has allowed different authors to generalize this notion to further homogeneous 3-dimensional manifolds of great interest: the hyperbolic 3-space \cite{bulo3,lrs}; the product spaces $\mathbb{H}^2\times\r$ \cite{bue1,bue2,bulo1,lipi} and $\s^2\times\r$ \cite{lm1}; the Heisenberg and the solvable group \cite{pip1,pip2}; and the special lineal group $SL(2,\r)$ \cite{lm2}. We emphasize that also conformal Killing vector fields have been considered, \cite{bulo2,moshs}, with the difference that the shape of the corresponding translators is not preserved along the flow. Finally, translators have been also addressed in Lorentzian spaces \cite{laor}.

Another useful characterization is due to Ilmanen \cite{ilm}, as he exhibited that translators are minimal surfaces in a conformal metric space. Furthermore, translators are also minimal surfaces when we endow $\r^3$ with weighted area and volume elements. Here we follow \cite{bcmr}: consider $e^\phi$ a smooth density, where $\phi\in C^\infty(\r^3)$, which serves as weight to measure the surface area and volume, $dA_\phi=e^\phi dA$ and $dV_\phi=e^\phi dV$. If we consider a compactly supported variation of $M$ with variation vector field $\xi$, then
$$
A'_\phi(0)=\int_MH_\phi\langle N,\xi\rangle dA_\phi,\quad V'_\phi(0)=-\int_M\langle N,\xi\rangle,
$$
where $H_\phi=H-\langle N,D\phi\rangle$ and $D$ is the gradient in $\r^3$. Thus, $M$ is a critical point of the weighted area for a compactly supported variation that preserves the weighted volume if and only if $H_\phi$ is a constant function, $H_\phi=\lambda$. If we drop the condition that the variations preserve the volume, then $M$ is a critical point of $A_\phi$ if and only if $H_\phi=0$. For the particular case $\phi(x)=\langle x,\vec{v}\rangle$ we have $D\phi=\vec{v}$ and thus the surfaces satisfying $H_\phi=0$ are the translators.

In contrast with the fruitful theory of translators developed in homogeneous 3-manifolds, a systematic study of equation $H_\phi=\lambda$, with $\lambda\neq0$, has been only considered in $\r^3$ \cite{blo,buor1,lop1,lop2}, in $\H^2\times\r$ \cite{buor2} and in $\mathbb{L}^3$ \cite{buor3}. Very recently, the author has addressed this problem in the product space $\s^2\times\r$ \cite{bue5}, where a classification of invariant surfaces of constant weighted mean curvature has been achieved. Following the aforementioned literature regarding translators in homogeneous 3-manifolds, we address the study of the following class of surfaces in the Heisenberg group.

\begin{defi}
Given $\lambda\in\r$, an oriented surface $M$ in $\nil$ is a $\lambda$-translator if its mean curvature $H$ satisfies
\begin{equation}\label{eqlambdatranslator}
H(p)=\langle N(p),\partial_z\rangle+\lambda,\qquad p\in M,
\end{equation}
where $N$ is the unit normal vector field of $M$.
\end{defi}

The Heisenberg group has a 4-dimensional isometry group, generated by three linearly independent left-translations and rotations about a vertical geodesic. The left-translations along a vertical geodesic will be referred as \emph{vertical translations}. In this paper, we provide a classification of invariant $\lambda$-translators under the following isometries of the Heisenberg group: left translations, rotations and helicoidal motions, being the latter defined as the composition of a rotation and a vertical translation. 

Our first goal is to classify invariant $\lambda$-translators under left-translations. Here, we distinguish if the left-translation is vertical or not. In analogy with the Euclidean case, we define a $\lambda$-grim reaper as a $\lambda$-translator invariant under a non-vertical left-translation. If the left-translation has a vertical component then it will be said to be \emph{tilted}. We first prove in Prop. \ref{propverticallambdatranslators} that the only $\lambda$-translators invariant under vertical translations are round cylinders. In our first result we assume that the left-translations are not vertical, and classify all $\lambda$-grim reapers.

\begin{teo}\label{t0}
For each $\lambda\in\r$, there exists a one-parameter family of $\lambda$-grim reapers, parametrized by the tilting. Its profile curve lies in a vertical plane, loops and self-intersects infinitely-many times.
\end{teo}

Our next goal is to obtain a classification of rotational $\lambda$-translators. Following the same scheme developed in \cite{bgm} which has been extended to several works by the author and further collaborators, we reach such a classification result by means of the study of the qualitative properties of the non-linear autonomous system fulfilled by the profile curve of any rotational $\lambda$-translator. This strategy or taking advantage of the symmetries to reduce the underlying PDE to an ODE system has been also exploited in \cite{fmp} in classifying symmetric minimal and positive constant mean curvature surfaces in $\nil$.

First of all, we prove in Prop. \ref{propintersectionorthogonal} that if a rotational $\lambda$-translator approaches the rotation axis, then it does orthogonally. Furthermore, in virtue of the existence of radial solutions to this problem in \cite{bue4}, we introduce in Def. \ref{defiM+-} two $\lambda$-translators, $M_+$ and $M_-$, that intersect orthogonally the rotation axis and whose unit normal at such intersection is $\partial_z$ and $-\partial_z$, respectively. The classification of rotational $\lambda$-translators will be done by distinguishing whether the intersection with the rotation axis occurs or not. First, we prove the following.

\begin{teo}\label{t1}
A complete, rotational $\lambda$-translator that intersects the rotation axis is one of the surfaces $M_\pm$ given in \ref{defiM+-}. Both $M_\pm$ are proper immersions of a disk and behave as follows:
\begin{enumerate}
\item $M_+$ is embedded and its end has third coordinate diverging to $\infty$. Furthermore, it converges to the vertical cylinder $\mathcal{C}_\lambda$ of CMC $\lambda$ and radius $1/(2\lambda)$.
\item $M_-$ has infinitely-many self intersections and its end has third coordinate diverging to $-\infty$. Furthermore, its distance to the rotation axis is unbounded. Moreover, assume that $p_0=(0,0,z_0)$ is the intersection of $M_-$ with the rotation axis. Then,
\begin{enumerate}
\item If $\lambda\geq1$ then $M_-$ is a strictly convex vertical graph around $p_0$ in the sense that $M_-$ lies locally above the plane $\{z=z_0\}$ around $p_0$ when this plane is endowed with downwards orientation.
\item If $\lambda<1$ then $M_-$ is a strictly concave vertical graph around $p_0$ in the sense that $M_-$ lies locally below the plane $\{z=z_0\}$ around $p_0$ when this plane is endowed with downwards orientation.
\end{enumerate}
\end{enumerate}
\end{teo}

For rotational $\lambda$-translators not-intersecting the rotation axis, we prove the following.

\begin{teo}\label{t2}
Let $M$ be a rotational $\lambda$-translator that does not intersect the rotation axis. Then, $M$ is either $\mathcal{C}_\lambda$, or $M$ is a properly immersed annulus with infinitely-many self-intersections. One end has its third coordinate diverging to $\infty$, is embedded and converges to $\mathcal{C}_\lambda$; the other end has its third coordinate diverging to $-\infty$, has infinitely-many self-intersections and its distance to the rotation axis is unbounded.
\end{teo}

Our final result concerns the classification of helicoidal $\lambda$-translators. Although the arguments for proving this result are mostly the same as the ones exhibited in the proof of Thm. \ref{t2}, the study of the associated non-linear autonomous system becomes substantially harder because of the existence of the vertical direction. Nonetheless, we overcome these difficulties and prove the following.

\begin{teo}\label{t3}
There exists a one-parameter family of helicoidal $\lambda$-translators, being exactly two whose generating curve intersect orthogonally the \emph{rotation} axis. Moreover,
\begin{enumerate}
\item The one with unit normal $\partial_z$ at the rotation axis has the same properties as the $\lambda$-translator $M_+$ of Thm. \ref{t1}.
\item The one with unit normal $-\partial_z$ at the rotation axis has the same properties as the $\lambda$-translator $M_-$ of Thm. \ref{t1}.
\item Any other $\lambda$-translator has a similar generating curve as any of the rotational $\lambda$-translators of Thm. \ref{t2}.
\end{enumerate}
\end{teo}

Finally, we detail the organization of the paper. In Section \ref{s2} we introduce the basic notation as well as the main properties of the Heisenberg space, including the computation of the mean curvature of an immersed surface. In Section \ref{s3} we deal with $\lambda$-grim reapers. First, we compute the mean curvature in terms of the coordinates of the base curve and express the $\lambda$-translator equation Eq. \eqref{eqlambdatranslator} in terms of such coordinate functions. In Section \ref{s31} we study the main properties of the phase plane of the non-linear autonomous system fulfilled by such coordinates, which will allow us to prove Thm. \ref{t0} in Section \ref{s32}. In Section \ref{s4} we study rotational $\lambda$-translators. In the same spirit as in Section \ref{s3}, we study structure of the phase plane of such rotational $\lambda$-translators in Section \ref{s41} and the main properties of its solutions in Section \ref{s42}. In Section \ref{s43} we prove the non-existence of rotational closed $\lambda$-translators, which will be useful for proving Thms. \ref{t1} and \ref{t2} in Sections \ref{s44} and \ref{s45}, respectively. Finally, in Section \ref{s5} we study helicoidal $\lambda$-translators. The main properties of the phase plane are depicted in Section \ref{s51}, while in Section \ref{s52} we prove Thm. \ref{t3}.

\section{Preliminaries}\label{s2}

The Heisenberg space $\nil$ is defined as $\r^3$ with usual coordinates $(x,y,z)$, endowed with the metric
\begin{equation}\label{eqmetricanil}
\langle\cdot,\cdot\rangle=dx^2+dy^2+(dz+\tau(ydx-xdy))^2.
\end{equation}
This space is a Lie group when we define the inner product
$$
(x_1,y_1,z_1)*(x_2,y_2,z_2)=(x_1+y_1,x_2+y_2,z_1+z_2+\tau(x_1y_2-x_2y_1)).
$$
For each $p_0=(x_0,y_0,z_0)\in\nil$ we define the \emph{left-translation} based on $p_0$ as
$$
L_{p_0}(x,y,z)=(x_0,y_0,z_0)*(x,y,z).
$$
The left-translations are isometries and in particular $\nil$ is a homogeneous space. If $e_1=\partial_x,\ e_2=\partial_y,\ e_3=\partial_z$ is the canonical basis of $\r^3$, then \begin{equation}\label{eqbasicvectorfields}
E_1=\partial_x-\tau y\partial_z,\quad E_2=\partial_y+\tau x\partial_z,\quad E_3=\partial_z
\end{equation}
are the only orthonormal (for the metric \eqref{eqmetricanil}), left-invariant vector fields satisfying $(E_k)_\textbf{0}=e_k$, where $\textbf{0}=(0,0,0)$. The left-translations and the vector fields $E_k$ are related as follows. The flow of isometries of the Killing vector field $E_3$ is the one-parameter group $t\mapsto L_{(0,0,t)}$, which are commonly known as the \emph{vertical translations} in the $\partial_z$-direction and will be denoted by $V_t$. The flow of isometries of $E_1$ is the one-parameter group $t\mapsto L_{(t,0,0)}$ while the one of $E_2$ is $t\mapsto L_{(0,t,0)}$. At this point, is worth-mentioning the remaining isometries of $\nil$. Besides the left-translations, we find the rotations about a vertical geodesic. If such a geodesic is the one passing through $\textbf{0}$, then the rotations are given by usual Euclidean rotations
$$
\rho_t:\nil\to\nil,\qquad \rho_t\left(\begin{matrix}
x\\y\\z
\end{matrix}\right)=\left(\begin{matrix}
\cos t&-\sin t &0\\
\sin t &\cos t&0\\
0&0&1
\end{matrix}\right)\left(\begin{matrix}
x\\y\\z
\end{matrix}\right).
$$
For instance, the flow of $E_1$ and of $E_2$ agree after a rotation, hence the left translations $L_{(t,0,0)}$ and $L_{(0,t,0)}$ differ by a rotation.

Since $\nil$ is not a space form and there are not homogeneous 3-manifolds with isometry group of dimension five, the aforementioned isometries describe the isometry group of $\nil$, being four-dimensional. Note that there are no reflection-type isometries and as a paramount consequence there is no an analog to the famous Alexandrov reflection technique. Finally, given a line $t\mapsto (tx_0,ty_0,z_0)$, the rotation of angle $\pi$ about this line is also an isometry.

\subsection{The mean curvature equation}

We compute an expression for the mean curvature of an immersed surface. Given $\psi=\psi(s,t):M\to\nil$ a parametric surface, we express the basic vector fields $\psi_s,\psi_t$ in the basis $\{E_k\}$,
$$
\psi_s=\sum_{k=1}^3 a_k(s,t)E_k(\psi(s,t)),\qquad \psi_t=\sum_{k=1}^3 b_k(s,t)E_k(\psi(s,t)).
$$
We drop the dependence on the variables $(s,t)$ in order to clarify the notation. The coefficients of the first fundamental form are $E=\langle\psi_s,\psi_s\rangle,\ F=\langle\psi_s,\psi_t\rangle$ and $G=\langle\psi_t,\psi_t\rangle$. A vector field orthogonal to $T_\psi M$ is
$$
\eta=(a_2b_3-a_3b_2)E_1(\psi)+(a_3b_1-a_1b_3)E_2(\psi)+(a_1b_2-a_2b_1)E_3(\psi),
$$
hence a unit normal vector field is $N=\eta/\sqrt{\langle\eta,\eta\rangle}$. The covariant derivatives are
\begin{align*}
\nabla_{\psi_s}\psi_s&=\sum_{k=1}^3\frac{\partial a_k}{\partial s}E_k(\psi)+a_k\sum_{j=1}^3a_j\nabla_{E_j(\psi)}E_k(\psi),\\
\nabla_{\psi_s}\psi_t&=\sum_{k=1}^3\frac{\partial b_k}{\partial s}E_k(\psi)+b_k\sum_{j=1}^3a_j\nabla_{E_j(\psi)}E_k(\psi),\\
\nabla_{\psi_t}\psi_t&=\sum_{k=1}^3\frac{\partial b_k}{\partial t}E_k(\psi)+b_k\sum_{j=1}^3b_j\nabla_{E_j(\psi)}E_k(\psi),
\end{align*}
thus the coefficients of the second fundamental form are
$$
e=\langle\nabla_{\psi_s}\psi_s,N\rangle,\quad f=\langle\nabla_{\psi_s}\psi_t,N\rangle,\quad  g=\langle\nabla_{\psi_t}\psi_t,N\rangle.
$$
For the computation of the covariant derivatives, we use that the Levi-Civita connection on the global basis \eqref{eqbasicvectorfields} is
\begin{equation}\label{eqLeviCivita}
\begin{aligned}
&\nabla_{E_1}E_1=0,& &\nabla_{E_2}E_1=-\tau E_3, &&\nabla_{E_3}E_1=-\tau E_2,\\
&\nabla_{E_1}E_2=\tau E_3,& &\nabla_{E_2}E_2=0,&&\nabla_{E_3}E_2=\tau E_1,\\
&\nabla_{E_1}E_3=-\tau E_2,& &\nabla_{E_2}E_3=\tau E_1,&&\nabla_{E_3}E_3=0.
\end{aligned}
\end{equation}
Finally, the mean curvature is 
$$
2H=\frac{Eg-2Ff+eG}{EG-F^2}.
$$

\section{$\lambda$-grim reapers}\label{s3}

We begin by investigating $\lambda$-translators that are invariant under left-translations, which are the analogs of cylindrical surfaces in $\r^3$. Recall that a surface in $\r^3$ is cylindrical if it is invariant by a one-parameter group of translations. If such translations have direction $\vec{v}\in\r^3$, then a cylindrical surface can be parametrized as $(s,t)\mapsto \alpha(s)+t\vec{v}$, where $\alpha(s)$ is a planar curve in some plane $\Pi$. In general, $\vec{v}$ and $\Pi$ are commonly chosen to be orthogonal, but there is no a priori relation between both of them. If $\vec{v}$ and $\Pi$ are not orthogonal, the cylindrical surface is said to be \emph{tilted}. 

Cylindrical translators in $\r^3$ are fully classified: up to a change of coordinates they are the grim-reaper cylinder, whose base curve in the $xz$-plane is the graph $z=-\log\cos x$, its tiltings and any vertical plane. For $\lambda\neq0$ a classification of cylindrical $\lambda$-translators in $\r^n$ was achieved independently in \cite{buor1,lop1}, where explicit parametrizations were obtained, and in $\mathbb{H}^2\times\r$ \cite{bulo1,lipi}. In the latter, such cylindrical surfaces were also referred to as grim reapers. A way to generate cylindrical surfaces in $\nil$ is by means of the submersion 
$$
\pi:\nil\to\r^2,\quad \pi(x,y,z)=(x,y),
$$
whose fibers, i.e. $\pi^{-1}(q),\ q\in\r^2$ are geodesics everywhere tangent to the vector field $\partial_z$. Given a curve $\alpha\subset\r^2$, the \emph{vertical cylinder over $\alpha$} is defined as $C_\alpha=\pi^{-1}(\alpha)$, which is everywhere tangent to $\partial_z$ and whose mean curvature is $H=\kappa_\alpha/2$, where $\kappa_\alpha$ is the curvature of $\alpha$. Conversely, any surface invariant under vertical translations must be a vertical cylinder $\pi^{-1}(\alpha)$, for some curve $\alpha$. In particular, the unit normal of any vertical cylinder is everywhere orthogonal to $\partial_z$. As a straightforward consequence, we deduce the following.

\begin{pro}\label{propverticallambdatranslators}
A $\lambda$-translator invariant by vertical translations is a vertical cylinder over a circle of constant curvature $\kappa_\alpha=2\lambda$.
\end{pro}
Note that $\r^2$ can be identified with the minimal plane $z=0$ (although not isometrically) and define analogously $C_\alpha=V_t(\alpha),\ t\in\r$. Under this analogy, we give the following definition.
\begin{defi}
A $\lambda$-grim reaper is a surface in $\nil$ invariant by a one-parameter group of non-vertical left-translations.
\end{defi}
The situation here changes as there are no totally geodesic surfaces in $\nil$ that are analogous to planes in $\r^3$. Nonetheless, we consider a $\lambda$-grim reaper as the image under a left-translation of a planar curve $\alpha(s)$ contained in some vertical plane. After a rotation, we assume that $\alpha(s)=(x(s),0,z(s))$ and translate $\alpha(s)$ under the one-parameter group of left-translations $t\mapsto L_{(0,t,ct)}$, where the parameter $c\geq0$ denotes the tilting. Only if $c>0$ we explicitly refer as a tilted $\lambda$-grim reaper. Thus, a parametrization of a $\lambda$-grim reaper is
\begin{equation}\label{eqparamcilindrica}
\psi(s,t)=L_{(0,t,ct)}(\alpha(s))=(x(s),t,-\tau t x(s)+z(s)+ct).
\end{equation}
We drop the dependence on the variable $s$ unless explicitly needed. The global orthonormal basis \eqref{eqbasicvectorfields} on $\psi$ is
$$
E_1(\psi)=(1,0,-\tau t),\quad E_2(\psi)=(0,1,\tau x),\quad E_3=(0,0,1),
$$
and the vector fields $\psi_s,\psi_t$ have coordinates in $E_k$
$$
\psi_s=x'E_1(\psi)+z'E_3(\psi),\quad \psi_t=E_2(\psi)+(c-2\tau x)E_3(\psi).
$$
Now we compute the covariant derivatives using \eqref{eqLeviCivita}. For the sake of clarity, we do explicitly the first one. Recall that $\psi_s=x'E_1(\psi)+z'E_3(\psi)$, hence $a_1=x',a_2=0$ and $a_3=z'$. We also drop the dependence of $E_k$ on $\psi$.
\begin{align*}
\nabla_{\psi_s}\psi_s=&\frac{\partial a_1}{\partial s}E_1+a_1\sum_{j=1}^3a_j\nabla_{E_j}E_1+\frac{\partial a_3}{\partial s}E_3+a_3\sum_{j=1}^3a_j\nabla_{E_j}E_3\\
=&x''E_1-2\tau x'z'E_2+z''E_3=(x'',-2\tau x'z',-\tau(2\tau xx'x'+tx'')+z'').
\end{align*}
An analog computation shows
$$
\nabla_{\psi_s}\psi_t=\tau z' E_1+\tau(2\tau x-c)x'E_2-\tau x'E_3,\qquad \nabla_{\psi_t}\psi_t=2\tau(c-2\tau x)E_1.
$$
The unit normal is
$$
N=\frac{1}{\sqrt{(1+(c-2\tau x)^2)x'^2+z'^2}}\left(-z'E_1-(c-2\tau x)x'E_2+x'E_3\right),
$$
whose denominator, $D$, is the square root of the determinant of the first fundamental form of $\psi$. The coefficients of the second fundamental form are
$$
e=\frac{z'(2\tau(c-2\tau x)x'^2-x'')+x'z''}{D},\quad f=\frac{\tau(-1+(c-2\tau x)^2)x'^2-\tau z'^2}{D},\quad g=\frac{-2\tau(c-2\tau x)z'}{D}.
$$
Thus, the mean curvature of a $\lambda$-grim reaper in terms of the coordinates of the base curve is
$$
2H=\frac{z'(-x''-(c-2\tau x)(-2\tau x'^2+(c-2\tau x)x''))+(1+c^2+4\tau x(-c+\tau x))x'z''}{((1+(c-2\tau x)^2)x'^2+z'^2)^{3/2}}.
$$
From the expression of the determinant of the first fundamental form, we consider the arc-length parameter
$$
x'=\frac{\cos\theta}{\sqrt{1+(c-2\tau x)^2}},\qquad z'=\sin\theta,
$$
where $\theta$ is a smooth function. With this arc-length parameter, the mean curvature and angle function of $\psi$ are
$$
2H=\sqrt{1+(c-2\tau x)^2}\theta',\qquad\langle N,E_3\rangle=\frac{\cos\theta}{\sqrt{1+(c-2\tau x)^2}},
$$
and from the $\lambda$-translator equation we conclude
$$
\theta'=\frac{2}{1+(c-2\tau x)^2}\left(\cos\theta+\lambda\sqrt{1+(c-2\tau x)^2}\right).
$$
We have proved that the coordinates of the base curve of a $\lambda$-grim reaper satisfy the ODE system
$$
\left\lbrace
\begin{array}{l}
\vspace{.2cm}x'=\dfrac{\cos\theta}{\sqrt{1+(c-2\tau x)^2}},\\
\vspace{.2cm}z'=\sin\theta,\\
\theta'=\dfrac{2}{1+(c-2\tau x)^2}\left(\cos\theta+\lambda\sqrt{1+(c-2\tau x)^2}\right).
\end{array}
\right.
$$
Conversely, for given initial conditions $(x_0,z_0,\theta_0)\in\r^3$, any solution to this initial value problem gives rise to a planar curve in the $xz$-plane which is the generating curve of a $\lambda$-grim reaper. Finally, note that the second function of this system is dependent on the first and the third, due to the invariance of both the mean curvature and angle function of a surface moving under vertical translations. Hence we study the solutions of this system by means of the first and the third equations.

\subsection{The phase plane of $\lambda$-grim reapers}\label{s31}

For each $c\geq0$, let us consider the non-linear, autonomous system
\begin{equation}\label{eqsystemcylindrical}
\left\lbrace
\begin{array}{l}
x'=\dfrac{\cos\theta}{\sqrt{1+(c-2\tau x)^2}},\\
\theta'=\dfrac{2}{1+(c-2\tau x)^2}\left(\cos\theta+\lambda\sqrt{1+(c-2\tau x)^2}\right).
\end{array}
\right.
\end{equation}
We define the phase plane of Eq. \eqref{eqsystemcylindrical} as the set $\R_c$ of the solutions $\gamma(s)=(x(s),\theta(s))$ of \eqref{eqsystemcylindrical}, also called orbits. Any orbit defines a curve $\alpha(s)=(x(s),0,z(s))$ which is the base curve for a $\lambda$-grim reaper. If $c=0$ we will simply denote $\R$ instead of $\R_0$.

A first geometric feature is the $2\pi$-periodicity in the $\theta$-direction that the phase plane has, hence we deduce the properties of $\R_c$ by analyzing any strip $\r\times[\theta_0,\theta_0+2\pi],\ \theta_0\in\r$. Note that in particular any two lines $\r\times\{\theta_0\}$ and $\r\times\{\theta_0+2\pi\}$ are identified. It is straightforward to see that if $\gamma(s)=(x(s),\theta(s))$ is an orbit, then $\sigma(s)=(\frac{c}{\tau}-x(-s),-\theta(-s))$ is also an orbit. Geometrically, $\R_c$ is anti-symmetric with respect to the lines $x=\frac{c}{2\tau}$ and $\theta=0$, and by $2\pi$-periodicity is also anti-symmetric with respect to $x=\frac{c}{2\tau}$ and $\theta=2k\pi,\ k\in\mathbb{Z}$.

Our first step is to study how an orbit moves throughout $\R_c$ and without losing any generality nor information, we restrict to the strip $\theta\in[0,2\pi]$. The $x$-coordinate increases for $\theta\in(0,\pi/2)\cup(3\pi/2,2\pi)$ and decreases for $\theta\in(\pi/2,3\pi/2)$. For the monotonicity of $\theta$, recall that if $\theta\in(0,\pi/2)\cup(3\pi/2,2\pi)$ then $\theta'>0$ and if $\theta\in(\pi/2,3\pi/2)$ then $\theta'<0$. Hence, we must study whether $\theta'=0$, or equivalently
$$
\cos\theta=-\lambda\sqrt{1+(c-2\tau x)^2}.
$$
First, note that for $\lambda\geq1$ this equation has no solutions, thus assume $\lambda<1$. For such values of $\lambda$, the above equation has solutions if and only if
$$
x\in\left[\frac{c}{2\tau}-\frac{\sqrt{1-\lambda^2}}{2\tau\lambda},\frac{c}{2\tau}+\frac{\sqrt{1-\lambda^2}}{2\tau\lambda}\right],
$$
with equality at the extremes if and only if $\theta=\pm\pi$. By defining 
$$
\Lambda_c(x)=\arccos(-\lambda\sqrt{1+(c-2\tau x)^2}),\quad x_\lambda=\frac{\sqrt{1-\lambda^2}}{2\tau\lambda},
$$ 
and the curve
\begin{equation}\label{eqcurvaLambda}
\Lambda_c=\{(x,\Lambda_c(x))\colon x\in[\frac{c}{2\tau}-x_\lambda,\frac{c}{2 \tau}+x_\lambda]\}\subset\R_c,
\end{equation}
the lines $\theta=\pi/2,3\pi/2$ and $\Lambda_c$ divide the strip $\theta\in(0,2\pi)$ into connected components where the coordinates of an orbit are strictly monotonous. In the region $\theta\in(\pi/2,\pi)$, the curve $\Lambda_c$ is a strictly convex graph, symmetric about the line $x=\frac{c}{2\tau}$ and that joins the points $(\frac{c}{2\tau}-x_\lambda,\pi)$ and $(\frac{c}{2\tau}+x_\lambda,\pi)$, having the point $(\frac{c}{2\tau},\arccos(-\lambda))$ as lowest point. By anti-symmetry and $2\pi$-periodicity, the curve $\Lambda_c$ is a strictly concave graph in the region $\theta\in(\pi,3\pi/2)$ with analog properties. We deduce that $\Lambda_c$ is indeed a symmetric bi-graph over the lines $\theta=(2k+1)\pi,\ k\in\mathbb{Z}$.

At first sight, we may think that we should study the orbits of $\R_c$ and their properties for each $c\geq0$. Nonetheless, we prove that given $c_1\neq c_2$, the phase planes $\R_{c_1}$ and $\R_{c_2}$ agree up to a translation in the $x$-direction. In the following result, we assume without losing generality $c_1=0$.

\begin{pro}
$\gamma(s)$ is an orbit in $\R$ if and only if $\sigma(s)=\gamma(s)+(\frac{c}{2\tau},0)$ is an orbit in $\R_c$.
\end{pro}

\begin{proof}
Assume that $\gamma(s)=(x(s),\theta(s))\in\R$ is an orbit and define $\sigma(s)=(x_c(s),\theta_c(s))=\gamma(s)+(\frac{c}{2\tau},0)$. Then, $x_c(s)=x(s)+\frac{c}{2\tau}$ and $\theta_c(s)=\theta(s)$ and it is immediate to check $\sigma$ is an orbit in $\R_c$. Conversely, if $\sigma$ is an orbit in $\R_c$, then it is straightforward that $\gamma=\sigma-(\frac{c}{2\tau},0)$ is an orbit in $\R$.
\end{proof}

In Fig. \ref{fig:phaseplanetilted}, top, we see the phase plane $\R_2$ for $c=2,\ \tau=1$ and $\lambda=0'5$, hence we translate all the elements of $\R$ by a unit in the $x$-direction. In red we have the curve $\Lambda$ for $c=0$, in green the curve $\Lambda_2$ and dashed in red the translation $\Lambda+(1,0)$ overlapping the curve $\Lambda_2$. In orange we see an orbit $\gamma$ for $c=0$ and in dashed the translation $\gamma+(1,0)$, overlapping the orbit $\sigma$ of $\R_2$ plotted in blue. At the bottom we see three $\lambda$-grim reapers with parameter $c=2$. Depending on the initial condition $x_0$ in Eq. \eqref{eqsystemcylindrical}, the tilting of the corresponding $\lambda$-grim reaper changes.

\begin{figure}[H]
\centering
\begin{tikzpicture}[scale=1]
\node[anchor=south,inner sep=0] at (0,0){\includegraphics[width=.4\textwidth]{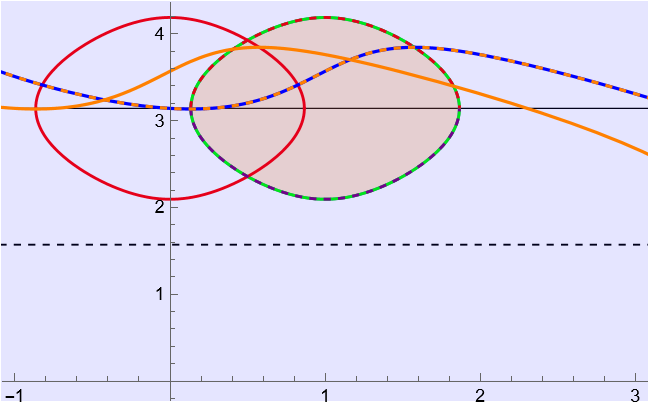}};
\node[anchor=north east,inner sep=0] at (-1.5,0){\includegraphics[width=.45\textwidth]{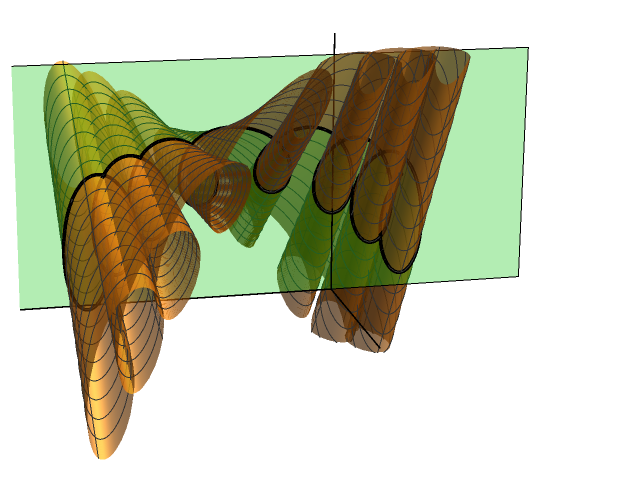}};
\node[anchor=north,inner sep=0] at (.4,0){\includegraphics[width=.45\textwidth]{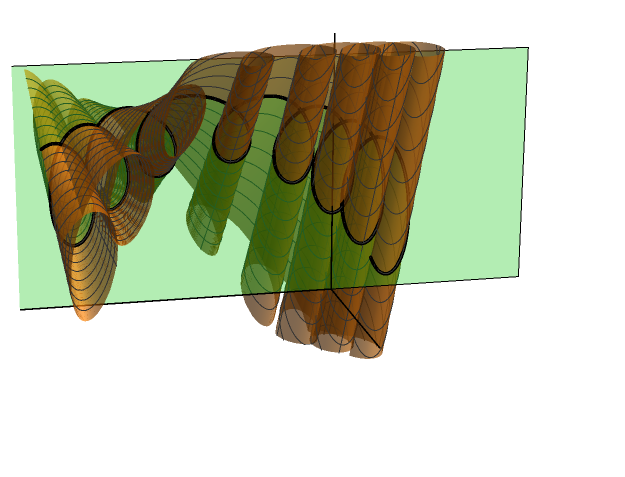}};
\node[anchor=north west,inner sep=0] at (2,0){\includegraphics[width=.45\textwidth]{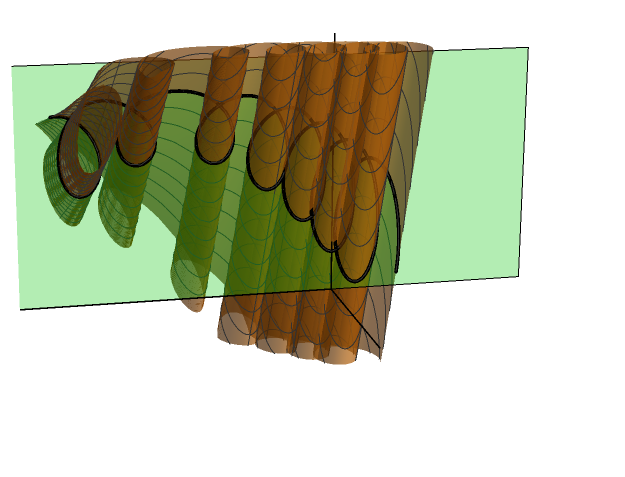}};
\end{tikzpicture}
\caption{Top: the phase plane of Eq. \eqref{eqsystemcylindrical} for $\tau=1$, $\lambda=0'5$ and $c=2$. In dashed there are the translations of the elements of the phase plane $\R$, overlapped with the ones of $\R_2$. Bottom: three $\lambda$-grim reapers for $c=2$ and different initial conditions $x_0$.}
\label{fig:phaseplanetilted}
\end{figure}

In virtue of this result, it is enough to study the properties of the orbits in a phase plane for a fixed parameter $c\geq0$ and deduce the properties of all the $\lambda$-grim reapers, independently of their tiltings.

\subsection{Proof of Thm. \ref{t0}}\label{s32}

As justified at the end of the previous section, we fix $c=0$ and study the properties of every $\lambda$-grim reaper by studying the orbits of $\R$. We now prove Thm. \ref{t0} by distinguishing between the cases $\lambda\geq1$ and $\lambda<1$.

First, assume $\lambda\geq 1$. In this case, the curve $\Lambda$ does not appear and we find essentially two types of monotonicity regions:
$$
\R_1=\left(0,\frac{\pi}{2}\right)\cup\left(\frac{3\pi}{2},2\pi\right),\quad \R_2=\left(\frac{\pi}{2},\frac{3\pi}{2}\right),
$$
and their $2\pi$-translations in the $\theta$-direction. We begin by proving that an orbit cannot lie contained in any of such monotonicity regions as $|s|\to\infty$. Arguing by contradiction, let $\gamma$ be an orbit and assume without losing generality that $\gamma(s)\in\R_1$ for $s\to\infty$; say in the region $\theta\in(0,\pi/2)$. By monotonicity, $\gamma(s)$ has some $\theta=\theta_0\in(0,\pi/2],$ as asymptote as $s\to\infty$. We express $\gamma$ as a vertical graph $(x,\theta(x))$, which satisfies
$$
\theta'(x)=\frac{\theta'(s)}{x'(s)}=\frac{2}{\cos\theta\sqrt{1+4\tau^2x^2}}\left(\cos\theta+\lambda\sqrt{1+4\tau^2x^2}\right).
$$
Note that there must exist a sequence $x_n\to\infty$ such that $\theta'(x_n)\to0$, but substituting above we get $\theta'(x_n)\to\frac{2\lambda}{\cos\theta_0}$, a contradiction.

Fix $x_0$ and consider the orbit $\gamma_0$ that passes through the point $(x_0,0)$ at $s=0$. For $s>0$, we know that $\gamma_0$ necessarily intersects the line $\theta=\pi/2$ at some $s_1>0$ and then $\gamma_0$ enters the region $\R_2$, with $\theta\in(\pi/2,\pi)$, until intersecting the line $\theta=3\pi/2$ at some $s_2>s_1$. Finally, $\gamma_0$ ends up intersecting the line $\theta=2\pi$ at some $(x_1,2\pi)$, at $s_3>s_2$. 

We prove next that $x_1<x_0$, and the proof consists on a series of claims. We begin by considering the system for the constant mean curvature $H=\lambda$,
$$
\left\lbrace
\begin{array}{l}
x'(s)=\dfrac{\cos\theta(s)}{\sqrt{1+4\tau^2x(s)^2}},\\
\theta'(s)=\dfrac{2\lambda}{\sqrt{1+4\tau^2x(s)^2}}.
\end{array}
\right.
$$
Let $\sigma=\sigma(t)=(x_\sigma,\theta_\sigma)$ be the solution of this system with initial conditions $\sigma(0)=(x_0,0)$. Recall that $\theta'_\sigma(t)\neq0$ and hence by the inverse function theorem we can express $\sigma$ as an horizontal graph $\sigma=x_\sigma(\theta)$. Furthermore, 
$$
\frac{dx_\sigma}{d\theta}=\dfrac{dx_\sigma/dt}{d\theta_\sigma/dt}=\frac{\cos\theta}{2\lambda},
$$
hence $\sigma$ can be explicitly integrated as $\sigma=(x_\sigma(\theta),\theta)$, where $x_\sigma(\theta)=\frac{\sin\theta}{2\lambda}+x_0$. In particular, $\sigma$ is a periodic orbit in the vertical direction.  Since $x'(0)=x_\sigma'(0)$ and $\theta'(0)>\theta_\sigma'(0)$, a comparison argument yields that $\gamma_0$ lies locally at the left-hand side of $\sigma$. Furthermore, this is fulfilled in the region $\theta\in(0,\pi/2)$, where $\gamma_0$ lies at the left-hand side of $\sigma$.

We claim that $\gamma_0$ and $\sigma$ cannot reach the line $\theta=\pi/2$ at the same point $(\hat{x},\pi/2)$. Arguing by contradiction, assume that both curves reach the same point for instants $\hat{s},\hat{t}>0$ and $x(s)<x_\sigma(t)$ for every $s\in(0,\hat{s})$ and $t\in(0,\hat{t})$. The horizontal graph $\sigma=x_\sigma(\theta)$ satisfies
$$
x_\sigma'(\pi/2)=x_\sigma'''(\pi/2)=0,\qquad x_\sigma''(\pi/2)=-\frac{1}{2\lambda}.
$$
Now, we express $\gamma_0$ as a horizontal graph $\gamma_0=x(\theta)$ whose derivative satisfies
$$
x'(\theta)=\frac{dx}{d\theta}=\frac{\cos\theta\sqrt{1+4\tau^2x^2(\theta)}}{2(\cos\theta+\lambda\sqrt{1+4\tau^2x^2(\theta)})}.
$$
This time, the derivatives of $x(\theta)$ at $\theta=\pi/2$ are
$$
x'(\pi/2)=0,\quad x''(\pi/2)=-\frac{1}{2\lambda},\quad x'''(\pi/2)=-\frac{1}{\lambda^2\sqrt{1+4\tau^2\hat{x}^2}}
$$
A Taylor expansion around $\theta=\pi/2$ yields that for $\theta<\pi/2$ close enough to $\pi/2$, the graph $x(\theta)$ is larger than the graph of $x_\sigma(\theta)$, a contradiction. Thus, $\gamma_0$ intersects $\theta=\pi/2$ at the left-hand side of $\sigma$, and stays there locally. See Fig. \ref{fig:fasescyllambdamayorigual1}, left.


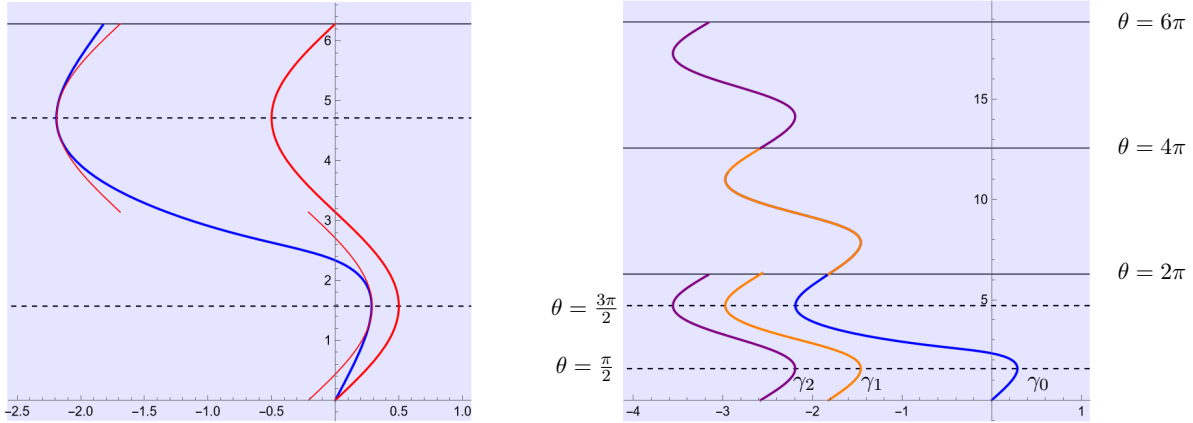
\begin{figure}[H]
\centering
\begin{tikzpicture}[scale=1]
\node[anchor=south east,inner sep=0] at (-1,0){\includegraphics[width=.4\textwidth]{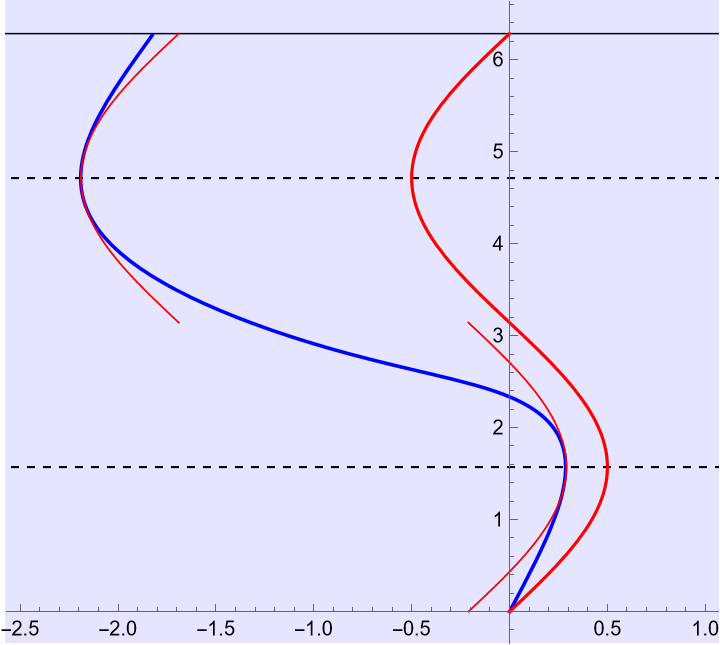}};
\node[anchor=south west,inner sep=0] at (1,0){\includegraphics[width=.4\textwidth]{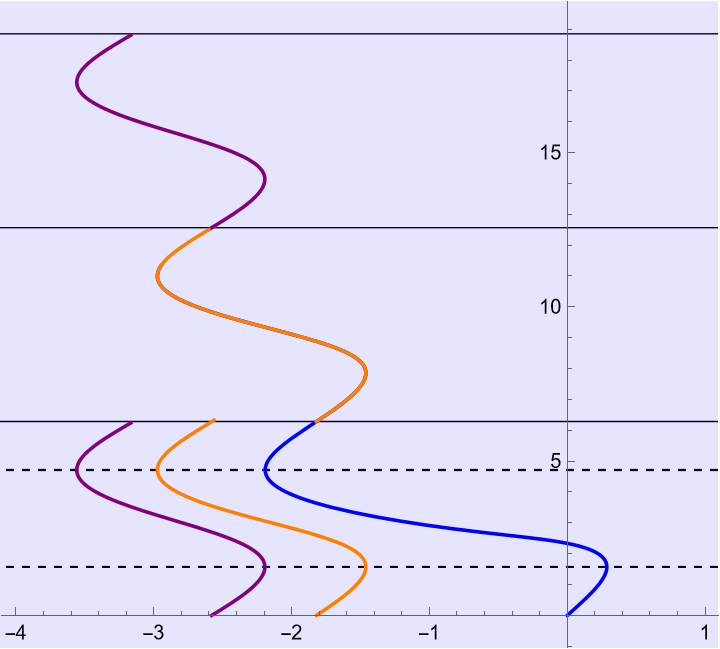}};
\draw (.5,.7) node[scale=.8]{$\theta=\frac{\pi}{2}$};
\draw (.5,1.5) node[scale=.8]{$\theta=\frac{3\pi}{2}$};
\draw (8,2) node[scale=.8]{$\theta=2\pi$};
\draw (8,3.65) node[scale=.8]{$\theta=4\pi$};
\draw (8,5.3) node[scale=.8]{$\theta=6\pi$};
\draw (6.5,.5) node[scale=.8]{$\gamma_0$};
\draw (4.3,.5) node[scale=.8]{$\gamma_1$};
\draw (3.4,.5) node[scale=.8]{$\gamma_2$};
\end{tikzpicture}
\caption{Left: the comparison between the orbits $\gamma_0$, in blue, and $\sigma$, in red, around a point at the line $\theta=\pi/2$. Right: the structure of the }
\label{fig:fasescyllambdamayorigual1}
\end{figure}

We can now prove that $x_1<x_0$. Arguing by contradiction, assume first that $x_1>x_0$. By continuity of the orbits and since $\sigma$ reaches again the point $(x_0,2\pi)$ by its periodicity, there must exist two instants $s_0,t_0$ such that $\gamma_0(s_0)=\sigma(t_0)$ and then $\gamma_0$ lies at the right-hand side of $\sigma$ for eventually reaching the point $(x_1,2\pi)$. Assume that such intersection occurs at the region $\theta\in(\pi/2,3\pi/2)$. Since $x'(s_0)=x_\sigma'(t_0)$, then necessarily $\theta'(s_0)>\theta'_\sigma(t_0)$ for $\gamma_0$ to cross $\sigma$. But substituting one has
$$
\theta'(s_0)=\frac{2\cos\theta(s_0)}{1+4\tau^2x(s_0)^2}+\frac{2\lambda}{\sqrt{1+4\tau^2x(s_0)^2}},\qquad\theta'_\sigma(t_0)=\frac{2\lambda}{\sqrt{1+4\tau^2x_\sigma(t_0)^2}}, 
$$
and since $x(s_0)=x_\sigma(t_0)$ since $\gamma_0(s_0)=\sigma(t_0)$ then $\theta'(s_0)<\theta'_\sigma(t_0)$, a contradiction. If the intersection occurs at the region $\theta\in(3\pi/2,2\pi)$ we arrive to a similar contradiction. The case where $\gamma_0$ and $\sigma$ both reach the line $\theta=3\pi/2$ at the same point is discarded in the same way as for the line $\theta=\pi/2$. This proves that $x_1\leq x_0$. Finally, recall that at $\theta=2\pi$ we have a similar comparison as at $\theta=0$, hence $\sigma$ lies locally at the left-hand side of $\gamma_0$ around $(x_0,2\pi)$ and in particular $\gamma_0$ cannot reach the point $(x_0,\pi/2)$, proving that necessarily $x_1<x_0$. 

Once we know the behavior of $\gamma_0$ in the strip $\theta\in(0,2\pi)$ and by $2\pi$-periodicity, the description of the orbits is now clear. Consider the orbit $\gamma_1$ passing through the point $(x_1,0)$. Since $\gamma_0$ and $\gamma_1$ cannot intersect by uniqueness, all the arguments above ensure that $\gamma_1$ has some $(x_2,2\pi)$ as endpoint, with $x_2<x_1$. Also, note that $\gamma_0$ and a translation of $\gamma_1$ can be smoothly glued together at the point $(x_1,2\pi)$, by just considering the vertical translation $\gamma_1+(0,2\pi)$. This forms a larger orbit, that will still be denoted by $\gamma_0$, that has the point $(x_2,4\pi)$ as endpoint. See Fig. \ref{fig:fasescyllambdamayorigual1}, right, the orbit in orange.

Now we consider the orbit $\gamma_2$ passing through $(x_2,0)$, which lies at the left-hand side of $\gamma_1$ by uniqueness, and that has some $(x_3,2\pi)$ as endpoint, with $x_3>x_2$. We glue again the larger orbit $\gamma_0$ with the translation $\gamma_2+(0,4\pi)$, see Fig. \ref{fig:fasescyllambdamayorigual1}, right, the orbit in purple. This process generates a strictly decreasing sequence $x_n$ and orbits $\gamma_n$ having $(x_n,0)$ and $(x_{n+1},2\pi)$ as endpoints. We claim that $x_n\to-\infty$. Arguing by contradiction, assume that $x_n\to x_\infty>-\infty$ and let $\gamma_\infty$ be the orbit that has $(x_\infty,0)$ and $(x_\infty^1,2\pi)$ as endpoints. Note that $x_\infty^1>x_\infty$, but we have $x_n\to x_\infty$ and $x_{n+1}\to x_\infty^1$, hence $x_\infty^1=x_\infty$ and we arrive to a contradiction. Thus, the $x$-coordinate of the orbit $\gamma_0$ is unbounded as $s\to\infty$. We have a similar situation for $s<0$ and we obtain a complete orbit $\gamma_0$ whose $x$-coordinate diverges as $|s|\to\infty$. The profile curve $\alpha(s)=(x(s),0,z(s))$ has unbounded $x$-coordinate and the $z$-coordinate increases and decreases, depending on if $\theta\in(0,\pi)$ or $\theta\in(\pi,2\pi)$ and their $2\pi$-translations, respectively. In particular, $\alpha$ has infinitely-many self-intersections. We refer again to Fig. \ref{fig:phaseplanetilted}, bottom.


Now, assume that $\lambda<1$. In this case, the curve $\Lambda$ given by Eq. \eqref{eqcurvaLambda} exists and the $\theta$-coordinate of an orbit $\gamma$ decreases whether $\gamma$ lies contained in the region $\Omega$ bounded by $\Lambda$. See Fig. \ref{fig:fasescyllambdamenor1}, left.

\begin{figure}[H]
\centering
\begin{tikzpicture}[scale=1]
\node[anchor=south east,inner sep=0] at (-1,0){\includegraphics[width=.4\textwidth]{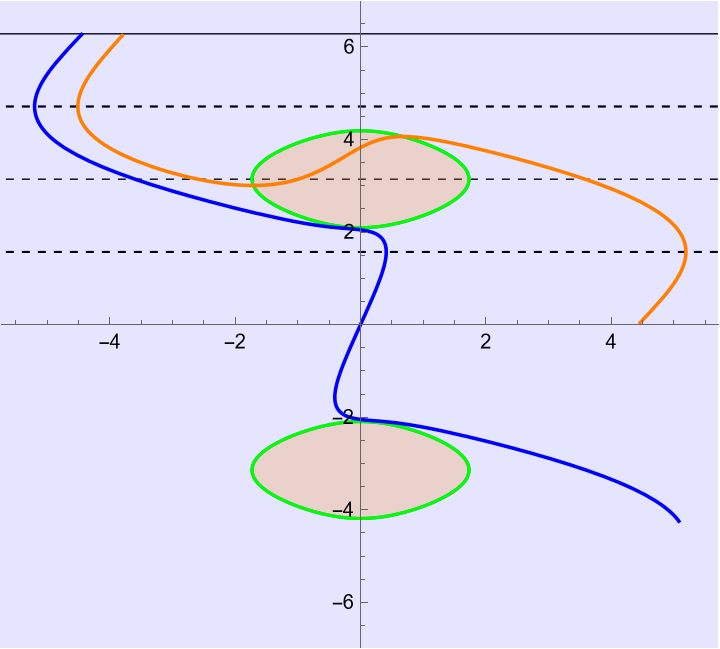}};
\node[anchor=south west,inner sep=0] at (1,-1.5){\includegraphics[width=.6\textwidth]{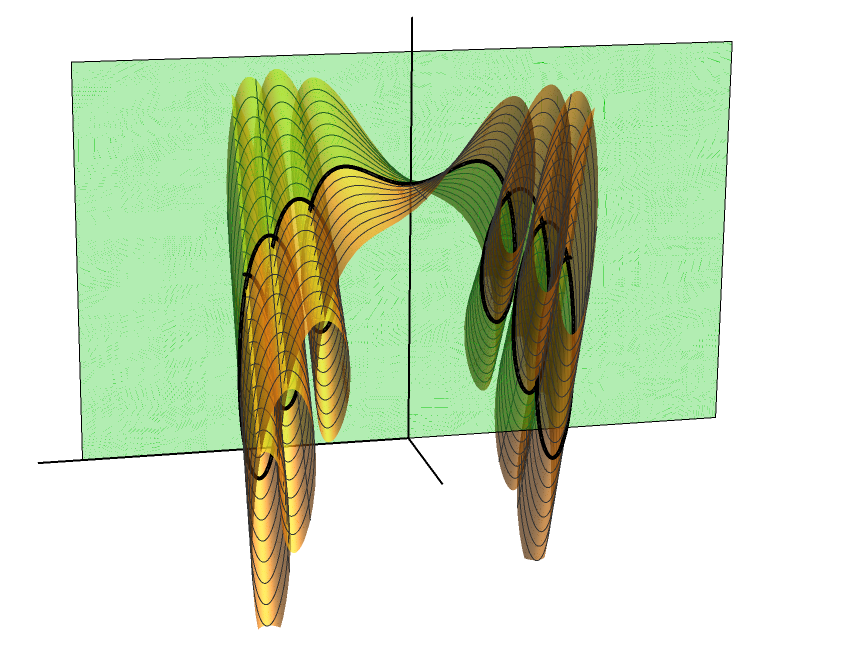}};
\draw (1,.8) node[scale=1]{$x$};
\draw (1,4.5) node[scale=1]{$y=0$};
\draw (6,.5) node[scale=1]{$y$};
\draw (5.75,5.75) node[scale=1]{$z$};
\end{tikzpicture}
\caption{The phase plane of Eq. \eqref{eqsystemcylindrical} for $\lambda<1$, where we have plotted the curve $\Lambda$ and two orbits following the motion in $\R$.}
\label{fig:fasescyllambdamenor1}
\end{figure}
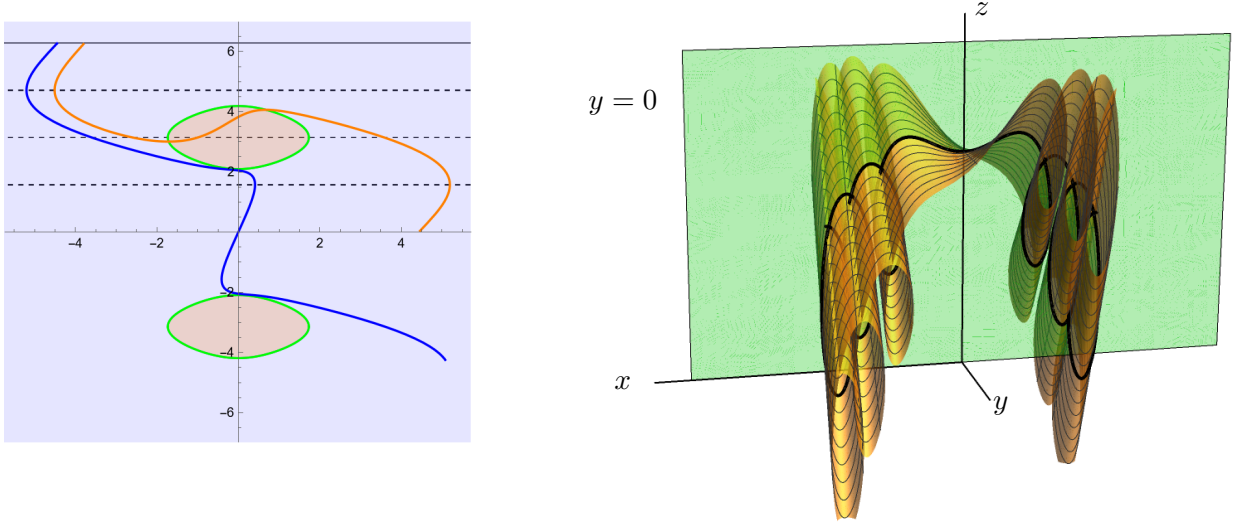

The properties of the orbits are similar as the ones previously deduced for $\lambda\geq1$. The only difference here is that an orbit may enter $\Omega$ and decrease its $\theta$-coordinate until leaving $\Omega$ again and then having the $\theta$-coordinate increasing. In any ways, this behavior can only occur a finite number of times since the $x$-coordinate diverges. Geometrically, the profile curve $\alpha(s)=(x(s),0,z(s))$ changes its convexity when its associated orbit $\gamma$ enters and leaves $\Omega$. See Fig. \ref{fig:fasescyllambdamenor1}, right.

\section{Rotational $\lambda$-translators}\label{s4}

In this section we focus on $\lambda$-translators invariant under a one-parameter group of rotations. Given a planar curve $\alpha(s)=(x(s),0,z(s))$, the rotational surface with \emph{base or profile curve} $\alpha(s)$ is parametrized by
$$
\psi(s,t)=(x(s)\cos t,x(s)\sin t,z(s)).
$$
A unit normal vector field is
$$
N=\frac{1}{\sqrt{(1+\tau^2x^2)x'^2+z'^2}}\left(-(\tau\sin t xx'+\cos tz')E_1+(\tau\cos t xx'-\sin tz')E_2+x'E_3\right),
$$
and the covariant derivatives are
\begin{align*}
\nabla_{\psi_s}\psi_s&=(2\tau\sin tx'z'+\cos tx'')E_1+(-2\tau\cos tx'z'+\sin tx'')E_2+z''E_3,\\
\nabla_{\psi_s}\psi_t&=(-\sin t(1+\tau^2x^2)x'+\tau\cos txz')E_1+(\cos t(1+\tau^2x^2)x'^2+\tau\sin txz')E_2-\tau xx'E_3,\\
\nabla_{\psi_t}\psi_t&=-\cos t x(1+2\tau^2x^2)E_1-\sin tx(1+2\tau^2x^2)E_2,
\end{align*}
hence we compute the coefficients of the second fundamental form
$$
e=\frac{-z'(2\tau^2 xx'^2+x'')+x'z''}{\sqrt{(1+\tau^2x^2)x'^2+z'^2}},\quad f=\frac{\tau x(\tau^2x^2x'^2-z'^2)}{\sqrt{(1+\tau^2x^2)x'^2+z'^2}},\quad g=\frac{x(1+2\tau^2x^2)z'}{\sqrt{(1+\tau^2x^2)x'^2+z'^2}}.
$$
Similarly to the case of $\lambda$-grim reapers, we consider the arc-length parameter
$$
x'=\frac{\cos\theta}{\sqrt{1+\tau^2x^2}},\qquad z'=\sin\theta,
$$
and the mean curvature and angle function of $\psi$ are
$$
2H=\frac{\sin\theta}{x}+\sqrt{1+\tau^2x^2}\theta',\qquad \langle N,E_3\rangle=\frac{\cos\theta}{\sqrt{1+\tau^2x^2}}.
$$
From Eq. \eqref{eqlambdatranslator} we conclude that the coordinates of the base curve of a rotational $\lambda$-translator fulfill the following ODE system
\begin{equation}\label{eqsystemrot0}
\left\lbrace
\begin{array}{l}
\vspace{.2cm} x'=\dfrac{\cos\theta}{\sqrt{1+\tau^2x^2}},\\
z'=\sin\theta,\\
\theta'=\dfrac{1}{\sqrt{1+\tau^2x^2}}\left(2\left(\dfrac{\cos\theta}{\sqrt{1+\tau^2x^2}}+\lambda\right)-\dfrac{\sin\theta}{x}\right).
\end{array}
\right.
\end{equation}
Conversely, given $(x_0,z_0,\theta_0)\in\r^3$ with $x_0>0$, there exists a unique solution to this system $\alpha(s)=(x(s),0,z(s))$ which generates a $\lambda$-translator under the image of the one-parameter family of rotations.

Note that system \eqref{eqsystemrot0} is singular at $x=0$, hence we cannot ensure the existence of rotational $\lambda$-translators intersecting the rotation axis by standard theory. Nevertheless, we prove that any $\lambda$-translator that approaches the rotation axis does it orthogonally. This has as consequence on the phase plane that an orbit cannot satisfy $x(s)\to0$ and $\theta(s)\to\theta_0\neq k\pi,\ k\in\mathbb{Z}$. 

\begin{pro}\label{propintersectionorthogonal}
If a rotational $\lambda$-translator intersects the rotation axis, then does it orthogonally.
\end{pro}

\begin{proof}
Let $\alpha(s)=(x(s),0,z(s))$ be the profile curve of a rotational surface whose coordinates satisfy \eqref{eqsystemrot0}. Assume that $x(s)\to0$ as $s\to s_0$; we can assume $s_0=0$ without losing generality. We have
\begin{align*}
(xz')'&=x'z'+xz''=x'z'+xx'\theta'\\
&=x'z'+\frac{xx'}{\sqrt{1+\tau^2x^2}}\left(2(x'+\lambda)-\frac{z'}{x}\right)\\
&=x'z'\left(1-\frac{1}{\sqrt{1+\tau^2x^2}}\right)+\frac{2xx'^2}{\sqrt{1+\tau^2x^2}}+\frac{2\lambda xx'}{\sqrt{1+\tau^2x^2}}.
\end{align*}
We integrate from $s=0$ to fixed $s$, obtaining
$$
x(s)\sin\theta(s)-\frac{4\lambda\tau^2x(s)^2}{2\tau^2(1+\sqrt{1+\tau^2x(s)^2})}=\int_0^sx'(t)z'(t)\left(1-\frac{1}{\sqrt{1+\tau^2x(t)^2}}\right)+\frac{2x(t)x'(t)^2}{\sqrt{1+\tau^2x(t)^2}}\,dt.
$$
For $s>0$ small enough we have $x(s)>0$ and we divide the above equation by $x(s)$,
$$
\sin\theta(s)-\frac{4\lambda\tau^2x(s)}{2\tau^2(1+\sqrt{1+\tau^2x(s)^2})}=\frac{1}{x(s)}\int_0^sx'(t)z'(t)\left(1-\frac{1}{\sqrt{1+\tau^2x(t)^2}}\right)+\frac{2x(t)x'(t)^2}{\sqrt{1+\tau^2x(t)^2}}\,dt.
$$
Letting $s\to0$ and applying the L'Hôpital rule in the right-hand side we get $\sin\theta(0)=0$, concluding $\theta(0)=k\pi,\ k\in\mathbb{Z}$ and therefore the intersection is orthogonal.
\end{proof}

The existence of orbits having a point $(0,k\pi)$ as finite endpoint follows from the existence of radial $\lambda$-translators intersecting orthogonally the rotation axis, see \cite{bue4}.

\begin{defi}\label{defiM+-}
We define $M_+$ and $M_-$ as the unique $\lambda$-translators intersecting orthogonally the rotation axis and whose unit normal at such intersection is $\partial_z$ and $-\partial_z$, respectively.
\end{defi}

\subsection{The phase plane of rotational $\lambda$-translators}\label{s41}

As in the case of $\lambda$-grim reapers, we observe that the second equation in \eqref{eqsystemrot0} can be obtained by means of the first and the third, hence we study the system.
\begin{equation}\label{eqsystemrotational}
\left\lbrace
\begin{array}{l}
\vspace{.2cm} x'=\dfrac{\cos\theta}{\sqrt{1+\tau^2x^2}},\\
\theta'=\dfrac{1}{\sqrt{1+\tau^2x^2}}\left(2\left(\dfrac{\cos\theta}{\sqrt{1+\tau^2x^2}}+\lambda\right)-\dfrac{\sin\theta}{x}\right).
\end{array}
\right.
\end{equation}
We define the phase plane $\R$ of \eqref{eqsystemrotational} as the set of orbits $\gamma(s)=(x(s),\theta(s))$, hence $\R=\{(x,\theta)\colon x>0,\theta\in\r\}$. Again, we see that $\R$ is $2\pi$-periodic in the $\theta$-direction, hence we can derive all the properties of $\R$ by restricting to any strip $\theta\in[\theta_0,\theta_0+2\pi],\ \theta_0\in\r$.

There exists a unique equilibrium $e_0=(1/(2\lambda),\pi/2)$, which corresponds to a vertical cylinder of constant mean curvature $\lambda$. The structure of $e_0$ is deduced by analyzing the corresponding linearized system. If we express \eqref{eqsystemrotational} as
$$
\left(\begin{matrix}
x\\ \theta
\end{matrix}\right)'=\left(\begin{matrix}
P(x,\theta)\\ Q(x,\theta)
\end{matrix}\right)=L(x,\theta),
$$
then 
$$
JL(e_0)=\left(
\begin{matrix}
0&-\dfrac{2\lambda}{\sqrt{4\lambda^2+\tau^2}}\\
\dfrac{8\lambda^3}{\sqrt{4\lambda^2+\tau^2}}&-\dfrac{8\lambda^2}{4\lambda^2+\tau^2}
\end{matrix}
\right),
$$
whose eigenvalues are
$$
\frac{4\lambda^2}{4\lambda^2+\tau^2}(-1\pm\sqrt{1-4\lambda^2-\tau^2}).
$$
Such eigenvalues are always distinct and either they are real and negative, or complex conjugate and of real negative part. In any case, $e_0$ is a stable node (in the first case) or spiral (in the second case) and the orbits close enough to $e_0$ converge to it as $s\to\infty$.

Next we turn our attention to the motion of the orbits in $\R$. We define the curve $\Gamma$ implicitly by
\begin{equation}\label{eqGamma}
\Gamma=\left\lbrace(x,\theta)\in\R\colon 2x\left(\dfrac{\cos\theta}{\sqrt{1+\tau^2x^2}}+\lambda\right)=\sin\theta\right\rbrace.
\end{equation}
The main difference here with the case of $\lambda$-grim reapers is that in the latter case the analog curve $\Lambda$ given by Eq. \eqref{eqcurvaLambda} was explicit, but in this case $\Gamma$ is given implicitly. This also differs from the study of the author for prescribed mean curvature surfaces in several ambient spaces, starting at \cite{bgm}.

The next result is of capital importance since it exhibits the main properties of the curve $\Gamma$, which is the one that determines the motion of any orbit in the phase plane. By the $2\pi$-periodicty of the phase plane in the $\theta$-direction, it suffices to restrict to the strip $\theta\in[0,2\pi]$.
\begin{teo}
Let us define
\begin{equation}\label{eqxlambda}
x_\lambda=\left\lbrace
\begin{array}{cll}
0&\mathrm{if}&\lambda\geq1,\\
\dfrac{\sqrt{1-\lambda^2}}{\lambda\tau}&\mathrm{if}&\mathrm{\lambda<1}.
\end{array}
\right.
\end{equation}
Then, $\Gamma$ is a horizontal graph $(x(\theta),\theta)$ for $\theta\in[0,\pi]$, that joins the points $(0,0)$ and $(x_\lambda,\theta)$, and there exists a unique $\theta_*\in(\pi/2,\pi)$ such that $x'(\theta_*)=0$, where $x$ reaches a global maximum. Furthermore,
\begin{itemize}
\item If $\lambda\geq1$, then $\Gamma$ only appears in the region $\theta\in[0,\pi]$.
\item If $\lambda<1$, then $\Gamma$ also appears in the region $[-\pi,-\pi/2)$ as a vertical graph $(x,\theta(x))$ for $x\in[0,x_\lambda]$, that joins the points $(0,-\pi)$ and $(x_\lambda,-\pi)$, and there exists a unique $x_*\in(0,x_\lambda)$ such that $\theta'(x_*)=0$, where $\theta$ reaches a global maximum. By $2\pi$-periodicity of $\R$, this component can be translated $(x,\theta(x))+(0,2\pi)$ and smoothly glued with the component of $\Gamma$ in $\theta\in[0,\pi]$, hence $\Gamma$ appears in the region $\theta\in(0,3\pi/2)$.
\end{itemize}
\end{teo}

\begin{proof}
We define
\begin{equation}\label{eqF}
F(x,\theta)=2x\left(\dfrac{\cos\theta}{\sqrt{1+\tau^2x^2}}+\lambda\right)-\sin\theta.
\end{equation}
Recall that $F(0,0)=0$ and
\begin{equation}\label{eqpartialFtheta}
\frac{\partial F}{\partial\theta}(x,\theta)=-\cos\theta-\frac{2x\sin\theta}{\sqrt{1+\tau^2x^2}},\qquad\frac{\partial F}{\partial\theta}(0,0)=-1.
\end{equation}
The implicit function theorem ensures the existence of $\epsilon>0$ and a function $\theta=\theta(x),\ x\in(-\epsilon,\epsilon)$, such that $\theta(0)=0$ and $F(x,\theta(x))=0$ for $x\in(-\epsilon,\epsilon)$ As a matter of fact, $\theta'(0)=2(\lambda+1)$ and thus $\theta(x)\in(0,\pi/2)$ for $x\in(0,\epsilon)$. Let us describe the behavior of the graph $(x,\theta(x))$. Note that $\theta'(x)=0$ if and only if
\begin{equation}\label{eqpartialFx}
\frac{\partial F}{\partial x}(x,\theta)=2\left(\lambda+\frac{\cos\theta}{(1+\tau^2x^2)^{3/2}}\right)=0,
\end{equation}
and thus $\theta(x)$ is strictly increasing in $\theta\in(0,\pi/2)$. From Eq. \eqref{eqF} we see that it cannot exist a sequence $(x_n,\theta_n)$ with $F(x_n,\theta_n)=0$ and $x_n\to\infty$, and thus the curve $\Gamma$ either intersects the line $\theta=\pi/2$ or stops being a graph $(x,\theta(x))$ at a finite point $(x_0,\theta_0)$, where $\theta'(x_0)=\infty$. Nonetheless, the latter cannot occur due to Eq. \eqref{eqpartialFtheta} and we conclude that $\Gamma$ actually intersects the line $\theta=\pi/2$ and in the full region $\theta\in(0,\pi/2]$ it is a vertical graph $(x,\theta(x))$ with $\theta'(x)>0$. 

If $\lambda<1$, then $F(x_\lambda,\pi)=0$ and thus $(x_\lambda,\pi)\in\Gamma$. Furthermore, at such point we have 
$$
\frac{\partial F}{\partial\theta}(x_\lambda,\pi)=1,\quad \frac{\partial F}{\partial x}(x_\lambda,\pi)=2\lambda(1-\lambda^2),
$$
and consequently $\Gamma$ is also locally a graph $(x,\theta(x))$ around $(x_\lambda,\pi)$ with $\theta'(x_\lambda)<0$.

We begin by proving that for any $\theta_0\in(0,\pi)$ there exists a unique $x_0>0$ such that $F(x_0,\theta_0)=0$. We prove the existence first. Recall that $F(0,\theta_0)=-\sin\theta_0<0$ and $\lim_{x\to\infty}F(x,\theta_0)=\infty$. Thus, there exists $x_0>0$ such that $F(x_0,\theta_0)=0$. Now we prove that such $x_0$ is unique by proving that in fact $\Gamma$ is a local horizontal graph $x=x(\theta)$ around each $(x_0,\theta_0)$. Indeed, we already know that $\theta'(x)$ never vanishes in the region $\theta\in(0,\pi/2]$ and thus the vertical graph $(x,\theta(x))$ can be also written as a horizontal graph $(x(\theta),\theta)$ for such a range of $\theta$. Hence, fix some $(x_0,\theta_0)\in\Gamma$ with $\theta_0\in(\pi/2,\pi)$. Since
$$
\frac{\partial F}{\partial x}(x_0,\theta_0)=2\left(\frac{\cos\theta_0}{(1+\tau^2x_0^2)^{3/2}}+\lambda\right),
$$
and substituting the fact that $F(x_0,\theta_0)=0$ we conclude
\begin{equation}\label{eqauxi1}
\frac{\partial F}{\partial x}(x_0,\theta_0)=\frac{\sin\theta_0}{2x_0}-\frac{2\tau^2x_0^2\cos\theta_0}{(1+\tau^2x^2)^{3/2}}>0.
\end{equation}
Again, the implicit function theorem ensures that $\Gamma$ is a local graph $x=x(\theta)$ around any $(x_0,\theta_0)$ and by continuity this graph is global.

Now we prove the existence and uniqueness of $\theta_*\in(\pi/2,\pi)$ such that $x'(\theta_*)=0$. Note that the existence is only clear when $\lambda\geq1$ since in such a case $\Gamma$ is a compact horizontal graph joining the points $(0,0)$ and $(0,\pi)$; but if $\lambda<1$ then $\Gamma$ is a horizontal graph $x=x(\theta)$ joining $(0,0)$ and $(x_\lambda,\pi)$, and it could happen that $x(\theta)$ was strictly increasing. We know that $x'(\theta)=0$ if and only if $\frac{\partial F}{\partial\theta}(x,\theta)=0$, which writes equivalently by Eq. \eqref{eqpartialFtheta} as
$$
\frac{\sqrt{1+\tau^2x^2}}{2x}=-\tan\theta.
$$
Hence the existence of $\theta_*\in(\pi/2,\pi)$ such that $x'(\theta_*)=0$ is assured whenever the graph of the function $-\arctan\frac{\sqrt{1+\tau^2x^2}}{2x}$ intersects the curve $\Gamma$. Note that the negative branch of the $\arctan$ function is identified with $\pi+\arctan$ due to the $\pi$-periodicity of the $\tan$ function. Hence, we define
\begin{equation}\label{eqf}
f(x)=\pi-\arctan\frac{\sqrt{1+\tau^2x^2}}{2x},\qquad x\geq0.
\end{equation}
Recall that $f$ extends continuously at $x=0$ by taking the value $f(0)=\pi/2$ and $f'(x)>0,\ \forall x>0$. By continuity and since $\Gamma$ is a compact horizontal graph joining $(0,0)$ and $(x_\lambda,\pi)$, the graph $(x,f(x))$ intersects $\Gamma$ at a point $(x_*,\theta_*)$ with $x_*=x(\theta_*)$ and $\theta_*\in(\pi/2,\pi)$. Furthermore, such $(x_*,\theta_*)$ is unique, since implicit derivation again yields $x''(\theta_*)<0$, hence $x$ reaches a local maximum at $\theta_*$ and then cannot intersect again $f(x)$. See Fig. \ref{fig:curveGamma}.

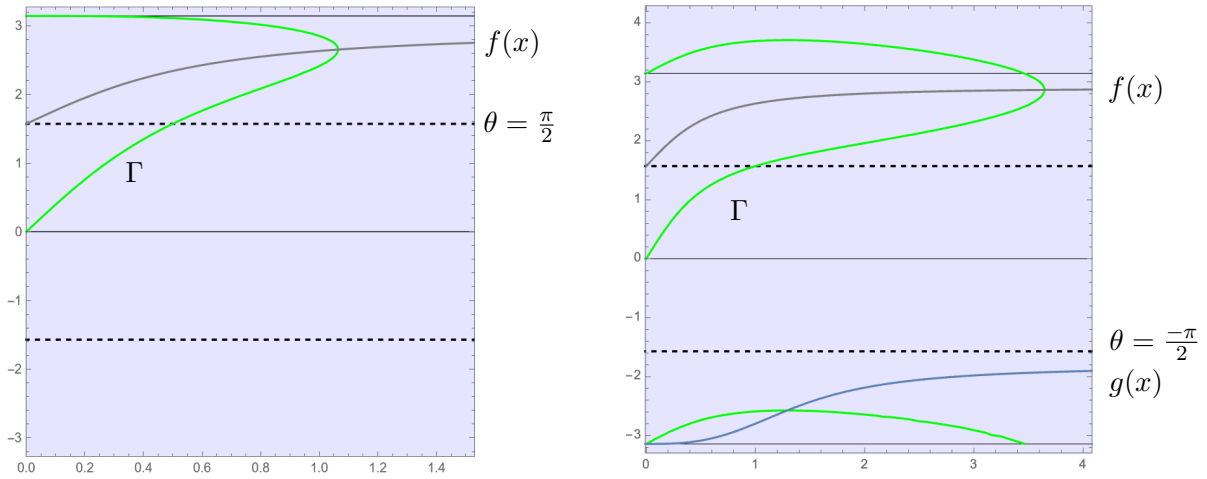
\begin{figure}[H]
\centering
\begin{tikzpicture}[scale=1]
\node[anchor=south east,inner sep=0] at (-1,0){\includegraphics[width=.4\textwidth]{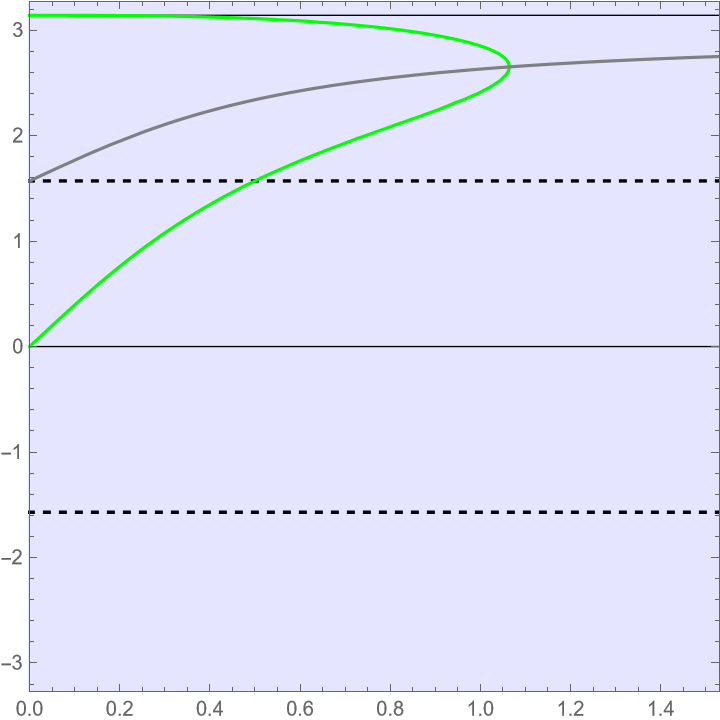}};
\node[anchor=south west,inner sep=0] at (1,0){\includegraphics[width=.4\textwidth]{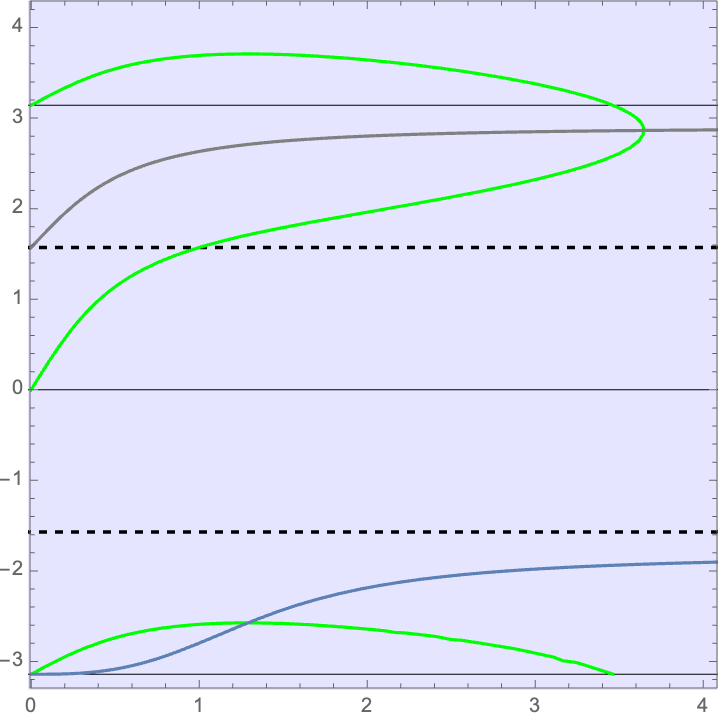}};
\draw (-.5,5.7) node[scale=1]{$f(x)$};
\draw (-.4,4.6) node[scale=1]{$\theta=\frac{\pi}{2}$};
\draw (-5.5,4) node[scale=1]{$\Gamma$};
\draw (2.5,3.5) node[scale=1]{$\Gamma$};
\draw (7.75,5.1) node[scale=1]{$f(x)$};
\draw (7.75,1.2) node[scale=1]{$g(x)$};
\draw (8,1.75) node[scale=1]{$\theta=\frac{-\pi}{2}$};
\end{tikzpicture}
\caption{Left: the curve $\Gamma$ in the phase plane for $\lambda\geq1$. Right: the curve $\Gamma$ in the phase plane for $\lambda<1$. The functions $f(x)$ and $g(x)$ determining the points of $\Gamma$ where the $x$-coordinate and the $\theta$-coordinate, respectively, reach local maxima have been also plotted.}
\label{fig:curveGamma}
\end{figure}

We finish by focusing on the case $\lambda<1$. We prove that the curve $\Gamma$ also appears for $\theta\in(-\pi,-\pi/2)$ as a vertical graph $\theta=\theta(x)$ that joins the points $(0,-\pi)$ and $(x_\lambda,-\pi)$. First, note that for $\theta\in(-\pi/2,0]$ the curve $\Gamma$ does not appear since $F(x,\theta)=0$ has no solution, hence we restrict when $\theta\in(-\pi,-\pi/2)$. Furthermore, when $\theta=-\pi$ we see that $x=0$ and $x=x_\lambda$ are the only solutions to this equation. At such points we have
$$
\frac{\partial F}{\partial\theta}(0,-\pi)=\frac{\partial F}{\partial\theta}(x_\lambda,-\pi)=1,
$$
hence the implicit function theorem ensures that around $(0,-\pi)$ and $(0,x_\lambda)$ we can express $\Gamma$ as vertical graphs $\theta=\theta_1(x)$ and $\theta=\theta_2(x)$, respectively. Consequently, $F(x,\theta)=0$ has two solutions for $\theta\in[-\pi,-\pi+\epsilon)$ for $\epsilon>0$ small enough. Geometrically, the lines $\theta=k$ intersect $\Gamma$ at two points for $k\in[-\pi,-\pi+\epsilon)$. Furthermore, for $\theta\in(-\pi,-\pi/2)$ we conclude from Eq. \eqref{eqpartialFtheta} that $\frac{\partial F}{\partial\theta}(x,\theta)>0$ hence $\theta_i(x)$ is never vertical. Consequently, $\Gamma$ can be expressed as a global vertical graph $\theta=\theta(x)$.

Next, we prove that there exists a unique $x_*\in(0,x_\lambda)$ such that $\theta'(x_*)=0$. This shows in particular that $\Gamma$ intersects the lines $\theta=\theta_0$ exactly twice for $\theta_0\in(-\pi,\theta_0)$, where $\theta_0=\theta(x_*)$. First, recall that $\theta'(x)=0$ if and only if
$$
0=\frac{\partial F}{\partial x}(x,\theta)=2\left(\lambda+\frac{\cos\theta}{(1+\tau^2x^2)^{3/2}}\right).
$$
Using that $F(x,\theta)=0$ we arrive to $\theta=\arctan\frac{2\tau^2x^3}{(1+\tau^2x^2)^{3/2}}$. Again, by $\pi$-periodicity of the $\tan$ function, the points of vanishing derivative of $\Gamma$ viewed as a graph $\theta=\theta(x)$ are located at the intersection with the graph of the function 
\begin{equation}\label{eqg}
g(x)=-\pi+\arctan\frac{2\tau^2x^3}{(1+\tau^2x^2)^{3/2}}.
\end{equation}
This function has vanishing derivative at $x=0$, and it is strictly increasing, hence its graph lies locally below $\theta(x)$ for $x>0$ small enough. By continuity, it intersects at least once the graph of $\theta(x)$ at some $(x_*,\theta(x_*))$, hence $\theta'(x_*)=0$. We compute the second derivative of $\theta$ by implicit derivation and use that $\theta'(x_*)=0$, obtaining
$$
\frac{\partial^2F}{\partial x^2}(x_*,\theta(x_*))+\frac{\partial F}{\partial\theta}(x_*,\theta(x_*))\theta''(x_*)=0.
$$
Since $\frac{\partial F}{\partial\theta}(x,\theta)>0$ for $\theta\in(-\pi,-\pi/2)$ and
$$
\frac{\partial^2F}{\partial x^2}(x,\theta)=-\frac{6x\tau^2\cos\theta}{(1+\tau^2x^2)^{5/2}}>0,
$$
we obtain $\theta''(x_*)<0$ and thus $\theta(x)$ has a local maximum at $x_*$. Therefore, the graphs of $\theta(x)$ and $g(x)$ can never intersect again and we have proved that $\theta(x)$ has a global maximum as unique critical point. We conclude that in the region $\theta\in(-\pi,-\pi/2)$, the curve $\Gamma$ is a vertical graph $\theta=\theta(x)$ that attains a global maximum and that joins the points $(0,-\pi)$ and $(x_\lambda,-\pi)$. By $2\pi$-periodicity, we can see this component in the region $(\pi,3\pi/2)$ and $\Gamma$ is a compact curve that joins the points $(0,0)$ and $(0,\pi)$. See Fig. \ref{fig:curveGamma}, right.
\end{proof}

\begin{defi}
We define the monotonicity regions of $\R$ as the connected components delimited by $\Gamma$ and $\theta=\pi/2+k\pi,\ k\in\mathbb{Z},$ where the components of any orbit are strictly monotonous functions.
\end{defi}

\subsection{The behavior of the orbits}\label{s42}

Once we have determined the structure of the phase plane and described its monotonicity regions, we focus on the motion and behavior of the orbits in it. The next result proves that an orbit cannot stay forever in a monotonicity region.

\begin{pro}
An orbit $\gamma(s)=(x(s),\theta(s))$ cannot stay in a monotonicity region with $x(s)\to\infty$ as $|s|\to\infty$.
\end{pro}

\begin{proof}
Arguing by contradiction, assume that $\gamma(s)=(x(s),\theta(s))$ is an orbit such that when $s\to\infty$, $\gamma(s)$ is strictly contained in some monotonicity region with $x(s)\to\infty$ (the case $s\to-\infty$ is proved similarly). In particular, $\theta(s)\neq\pi/2+k\pi$ for any $k\in\mathbb{Z}$; without losing generality, assume that $\gamma(s)\subset\{\theta\in(-\pi/2,\pi/2)\}$. By assumption,  $\gamma$ can be expressed as a vertical graph $\theta=\theta(x)$ for $x\geq x_0$ and some $x_0>0$. The inverse function theorem yields
$$
\theta'(x)=\frac{d\theta}{dx}=\dfrac{d\theta/ds}{dx/ds}=\frac{1}{\cos\theta(x)}\left(2\left(\frac{\cos\theta(x)}{\sqrt{1+\tau^2x^2}}+\lambda\right)-\frac{\sin\theta(x)}{x}\right),
$$
and in this region we have $\theta'(x)>0$ for any $x>x_0$. Since $\theta(x)$ is strictly increasing and bounded from above by $\theta=\pi/2$, there exists a sequence $x_n\to\infty$ for which $\theta(x_n)\to\theta_0\in(-\pi/2,\pi/2]$ and $\theta'(x_n)\to0$. Substituting above yields $\theta'(x_n)\to 2\lambda/\cos\theta_0\neq0$, which is a contradiction.
\end{proof}

Next, we prove the non-existence of closed orbits in $\R$, and its proof follows from Bendixson-Dulac theorem.

\begin{pro}\label{propnoclosedorbits}
There do not exist closed orbits in $\R$.
\end{pro}

\begin{proof}
Let us express Eq. \eqref{eqsystemrotational} as $(x,\theta)'=(P(x,\theta),Q(x,\theta))$ and define on $\R$ the vector field $V(x,\theta)=x\sqrt{1+\tau^2x^2}(P(x,\theta),Q(x,\theta))$. Recall that if $\gamma$ is an orbit, then $V(\gamma)$ is tangent to $\gamma$ at each point. A simple computation shows that
$$
\mathrm{div}V=-\frac{2x\sin\theta}{\sqrt{1+\tau^2x^2}}<0.
$$
Assume by contradiction that $\gamma$ is a closed orbit in $\R$, name $\Omega$ to its inner region and let $\textbf{n}_\gamma$ be its unit normal pointing towards $\Omega$. The divergence theorem yields
$$
0>\int_\Omega\mathrm{div}V=-\int_\gamma\langle\textbf{n}_\gamma,V\rangle=0.
$$
This contradiction proves the result.
\end{proof}

The non-existence of closed orbits will be exploited with the consequences of Poincaré-Bendixson theorem, a powerful result that gives sufficient conditions for the existence of closed orbits in the phase plane of an autonomous system. We announce it next for the reader's convenience.

\begin{teo}[Poincaré-Bendixson]
Let $K$ be a compact region of the phase plane of a non-linear autonomous system. Assume that $K$ contains finitely-many equilibrium points and let $\gamma(s)$ be an orbit for which there exists $s_0\in\r$ such that $\gamma(s)\in K,\ \forall s\geq s_0$. Then, either 
\begin{enumerate}
\item $\gamma(s)$ converges to an equilibrium as $s\to\infty$;
\item $\gamma$ is itself a closed orbit; 
\item or $\gamma$ spirals towards a closed orbit as $s\to\infty$. 
\end{enumerate}
\end{teo}

In our setting and in virtue of Prop. \ref{propnoclosedorbits} and the existence of a single equilibrium $e_0$, we conclude the following.

\begin{cor}\label{corcomportamientooribtasPB}
Let be $K\subset\R$ a compact set and $\gamma$ a solution of \eqref{eqsystemrotational}. Let $s_0\in\r$ such that $\gamma(s_0)\in K$.
\begin{enumerate}
\item If $\gamma(s)\in K$ for every $s\geq s_0$, then $e_0\in K$ and $\gamma(s)\to e_0$ as $s\to\infty$.
\item If $e_0\notin K$, there exists $s_+>s_0$ such that $\gamma(s_+)\notin K$.
\item Regardless of whether $K$ contains $e_0$ or not, there exists $s_-<s_0$ such that $\gamma(s_-)\notin K$.
\end{enumerate}
\end{cor}

\begin{proof}
The first two items are clear from Poincaré-Bendixson theorem in our setting. For the last one, if $\gamma(s)\in K$ for every $s<0$, then the non-existence of closed orbits yields that $\gamma(s)\to e_0$ as $s\to-\infty$, contradicting the stable node or spiral structure of $e_0$.
\end{proof}
We finish this section by announcing a technical result that will be useful in the sequel and that compares the $x$-coordinate of the endpoints of a compact orbit in a certain region of $\R$.

\begin{pro}\label{propextremosorbita}
Let $\gamma$ be an orbit in the region $\theta\in[0,\pi]$  with endpoints $(x_1,0)$ and $(x_2,\pi)$. Then, $x_2<x_1$.
\end{pro}

\begin{proof}
We can assume without losing generality that $x_2>x_\lambda$, where $x_\lambda$ was defined in Eq. \eqref{eqxlambda}, since on the contrary $\gamma$ must intersect the curve $\Gamma$ and then intersect $\theta=\pi$ at some $(\hat{x_2},\pi)$ with $\hat{x_2}>x_2$. In such a case, we conclude $x_1>\hat{x_2}>x_2$.

We know that $\gamma$ can be expressed as a horizontal graph $x=x(\theta)$ with $x:[0,\pi]\to\r$ such that $x(0)=x_1$ and $x(\pi)=x_2$. Then,
\begin{equation}\label{eqderivadax}
x'(\theta)=\dfrac{\cos\theta}{2\left(\dfrac{\cos\theta}{\sqrt{1+\tau^2x(\theta)^2}}+\lambda\right)-\dfrac{\sin\theta}{x(\theta)}}.
\end{equation}
We decompose the function $x$ as
$$
\textbf{x}_1:[0,\pi/2]\to\r,\qquad \textbf{x}_2:[\pi/2,\pi]\to\r,
$$
and define $\hat{\textbf{x}_1}(\theta)=\textbf{x}_1(\pi-\theta)$. Recall that the graph $(\hat{\textbf{x}_1}(\theta),\theta)$ agrees with the one of $\textbf{x}_1$ up to a reflection about the line $\theta=\pi/2$.
Let $\epsilon>0$ and define $\theta_\epsilon=\pi/2+\epsilon$. Then,
$$
\hat{\textbf{x}_1}'(\theta_\epsilon)=\dfrac{-\sin\epsilon}{2\left(\dfrac{\sin\epsilon}{\sqrt{1+\tau^2\hat{\textbf{x}_1}(\theta_\epsilon)}}+\lambda\right)-\dfrac{\cos\epsilon}{\hat{\textbf{x}_1}(\theta_\epsilon)}}, \quad \textbf{x}_2'(\theta_\epsilon)=\dfrac{-\sin\epsilon}{2\left(\dfrac{-\sin\epsilon}{\sqrt{1+\tau^2\textbf{x}_2(\theta_\epsilon)}}+\lambda\right)-\dfrac{\cos\epsilon}{\textbf{x}_2(\theta_\epsilon)}}.
$$
For $\epsilon\to0$ we know that $\hat{\textbf{x}_1}(\theta_\epsilon)$ and $\textbf{x}_2(\theta_\epsilon)$ both converge to the value $x(\pi/2)$, hence for $\epsilon$ small enough we have $\hat{\textbf{x}_1}'(\theta_\epsilon)>\textbf{x}_2'(\theta_\epsilon)$, which yields that the graph of $\hat{\textbf{x}_1}$ is locally above the graph of $\textbf{x}_2$ for $\epsilon$ small enough.

Now, we argue by contradiction by assuming $x_2\geq x_1$. If $x_2>x_1$ this means that there exists a largest $\theta_*\in(\pi/2,\pi)$ such that $\hat{\textbf{x}_1}(\theta_*)=\textbf{x}_2(\theta_*)=x_*$ and the graph of $\hat{\textbf{x}_1}$ is transverse to the graph of $\textbf{x}_2$ in such a way that $\hat{\textbf{x}_1}(\theta)<\textbf{x}_2(\theta)$ for every $\theta\in(\theta_*,\pi)$. In particular, $\hat{\textbf{x}_1}'(\theta_*)<\textbf{x}_2'(\theta_*)$. This is however a contradiction with Eq. \eqref{eqderivadax}, since
$$
\hat{\textbf{x}_1}'(\theta_*)=\dfrac{\cos\theta_*}{2\left(\dfrac{-\cos\theta_*}{\sqrt{1+\tau^2x_*^2}}+\lambda\right)-\dfrac{\sin\theta_*}{x_*}}>\dfrac{\cos\theta_*}{2\left(\dfrac{\cos\theta_*}{\sqrt{1+\tau^2x_*^2}}+\lambda\right)-\dfrac{\sin\theta_*}{x_*}}=\textbf{x}_2'(\theta_*),
$$
since $\cos\theta_*<0$ and both denominators are positive since $\gamma$ lies at the right-hand side of $\Gamma$.

If $x_2=x_1$ then we know that necessarily $\hat{\textbf{x}_1}(\theta)>\textbf{x}_2(\theta)$ for every $\theta\in[\pi/2,\pi)$ until reaching the instant $\theta=\pi$ where they agree. Consequently, a comparison argument yields $\hat{\textbf{x}_1}'(\pi)\leq\textbf{x}_2'(\pi)$, but at $\theta=\pi$ we have
$$
\hat{\textbf{x}_1}'(\pi)=\dfrac{-1}{\dfrac{2}{\sqrt{1+\tau^2x_1^2}}+\lambda}>\dfrac{-1}{\dfrac{-2}{\sqrt{1+\tau^2x_1^2}}+\lambda}=\textbf{x}_2'(\pi).
$$
\end{proof}

A similar argument allows us to prove the analog in the region $\theta\in[-\pi,0]$. Combining these facts, we conclude the following.

\begin{cor}\label{cororbitadescapa}
Let $\gamma$ be an orbit in the region $\theta\in[-\pi,\pi]$ with endpoints $(x_0,-\pi)$ and $(x_2,\pi)$, and that passes through $(x_1,0)$. Then, $x_2<x_1<x_0$.
\end{cor}

\subsection{Non-existence of rotational closed $\lambda$-translators}\label{s43}

In this section we prove the non-existence of rotational closed $\lambda$-translators, as stated next.

\begin{teo}\label{thmnonexistenceclosed}
There do not exist closed rotational $\lambda$-translators.
\end{teo}

Before proving the result, we think that few comments must be done. The existence of certain closed surfaces in $\nil$ is a challenging task even for the widely studied case of CMC surfaces. For instance, the uniqueness of the rotational embedded CMC spheres among closed, embedded surfaces in $\nil$ is still an outstanding open problem. In \cite{bue3}, the author generalized the classification of rotational CMC surfaces to the case when the mean curvature depends on the unit normal. It was proved that for certain choices of the prescribed function, the rotational embedded prescribed mean curvature spheres constructed were not unique in the Alexandrov sense, as it was exhibited the existence of embedded tori which were even rotational. Note that rotational CMC tori do not exist in $\nil$ in virtue of the classification in \cite{tom}.

Regarding the non-existence of closed $\lambda$-translators, in $\r^3$ and in $\mathbb{H}^2\times\r$ it is a consequence of the elliptic equation $\Delta h=2H\langle N,\partial_z\rangle$ fulfilled by the restriction of \emph{height function} $h$ to an immersed surface, and the fact that any closed surface is null-homologous in the second homology group of the ambient space, hence is the boundary of a 3-chain. Then we conclude by Stoke's theorem; see e.g. \cite{buor2,lop2}. The product structure in both spaces makes this height function readily available, but there is not an analog to this function in $\nil$. This has as consequence the non-existence of sharp height estimates for CMC graphs, which differs from the scenario in $\r^3$ and the product spaces $\mathbb{M}^2\times\r$.

Although the non-existence of a general closed $\lambda$-translator with any further assumption on its symmetries or topology seems to the author an interesting, and far for being trivial, problem to address, we restrict ourselves to the non-existence of rotational ones since the main purpose of the paper is to study invariant $\lambda$-translators and in particular rotational ones. Note that any closed rotational surface in $\nil$ must have genus 0 or 1, since $\nil$ is diffeomorphic to $\r^3$ and a connected planar curve cannot introduce multiple disjoint cycles needed for adding higher genus. Thus, a rotational closed $\lambda$-translator must be either diffeomorphic to a sphere or a torus. After all the background regarding closed surfaces, we prove Thm. \ref{thmnonexistenceclosed}.

\begin{proof}
The proof is by contradiction. Let $\alpha(s)=(x(s),0,z(s))$ be the base curve of a rotational closed $\lambda$-translator. Thus, $\alpha(s)$ either stays at a positive distance to the rotation axis and is a closed curve (hence generating a torus), or intersects the rotation axis orthogonally (hence generating a sphere).

First, assume that $\alpha(s)$ generates a torus and thus its image is a closed curve, possibly with self-intersections. Thus, there exists a minimum $T>0$ such that $\alpha(s)=\alpha(s+T)$ for every $s\in\r$. In particular, both the $x$ and $z$-coordinates are $T$-periodic functions. Let us denote by $\gamma(s)=(x(s),\theta(s))$ to the orbit corresponding to $\alpha(s)$ in a fundamental interval $[s_0,s_0+T]$; the closeness of $\alpha(s)$ implies that $\gamma(s)$ is a periodic orbit. In case that $\alpha(s)$ does not self-intersect before closing, then after a translation of $2\pi$ in the $\theta$-direction we have that $\gamma(s)$ lies contained in the region $\theta\in[-\pi,\pi]$. However, this is impossible due to Cor. \ref{cororbitadescapa}. If $\alpha(s)$ self-intersects then it does a finite number of times and thus $\gamma(s)$ is again a periodic orbit but this time lies contained in some strip $\theta\in[(2k-1)\pi,(2m+1)\pi]$, with $k<m$. In any case, we contradict again Cor. \ref{cororbitadescapa} after invoking it a finite number of times. This concludes the non-existence of rotational $\lambda$-translators with the topology of a torus.

If $\alpha(s)$ generates a sphere, the proof is similar. Let be $\gamma(s)$ the orbit corresponding to $\alpha(s)$ and assume by contradiction that they generate a $\lambda$-translator $S_0$ with the topology of a sphere. Let $p_0,q_0\in S_0$ the points of intersection of $S_0$ with the rotation axis. After a reparametrization of $\alpha(s)$ we assume $\alpha(0)=p_0$ and $\alpha(s_0)=q_0$ with $s_0>0$ and thus $x(s)$ increases for $s>0$ small enough and decreases for $s<s_0$ close enough to $s_0$. In particular, $x'(0)=1$ and $x'(s_0)=-1$. This implies that the orbit $\gamma(s)$ is a compact arc having as endpoints $(0,2k\pi)$ and $(0,(2m+1)\pi)$, with $k,m\in\mathbb{Z}$. We can assume, up to a discrete translation in the $\theta$-direction of $\R$, that $\gamma(0)=(0,0)$ and $\gamma(s_0)=(0,(2m+1)\pi)$ with $m\geq0$. 

Start by assuming that $m=0$ and thus $\gamma$ has as endpoints $(0,0)$ and $(0,\pi)$ and passes through some $(x_0,\pi/2)$ with $x_0>0$. Thus, $S_0$ is embedded since $z'(s)>0$ for every $s\in(0,s_0)$. Let us decompose $\gamma$ as $\gamma_1\cup\gamma_2$, where $\gamma_1$ is the portion of $\gamma$ in the region $\theta\in[0,\pi/2]$ and $\gamma_2$ is the portion of $\gamma$ in the region $\theta\in[\pi/2,\pi]$. From Prop. \ref{propextremosorbita} we know that $\gamma_2$ should intersect $\theta=0$ at some $(x_1,0)$ with $x_1>0$. But this is actually impossible since $\gamma_2$ in the region $\theta\in[0,\pi/2]$ agrees with $\gamma_1$. This contradiction proves that there does not exist an orbit having the points $(0,0)$ and $(0,\pi)$ as endpoints. If $m\geq1$ the proof is similar after applying Cor. \ref{cororbitadescapa} a finite number of times. 

This proves the non-existence of closed rotational $\lambda$-translators.
\end{proof}

\subsection{Proof of Thm. \ref{t1}}\label{s44}

Next we prove Thm. \ref{t1} by classifying the rotational $\lambda$-translators intersecting the rotation axis. By Prop. \ref{propintersectionorthogonal}, we know that such intersection must be orthogonal, and by Def. \ref{defiM+-} we know the existence of rotational $\lambda$-translators $M_\pm$ intersecting orthogonally the rotation axis. 

Let $\gamma_\pm$ the orbits corresponding to the $\lambda$-translators $M_\pm$, respectively, as defined in Def. \ref{defiM+-}. We know that $\gamma_+$ has the point $(0,0)$ as endpoint, say at $s=0$, and then for $s>0$ the monotonicity of $\R$ implies that $\gamma_+$ lies at the right-hand side of $\Gamma$. Similarly, $\gamma_-$ has $(0,\pi)$ as endpoint, again say at $s=0$, but this time we must distinguish between the values $\lambda\geq1$ and $\lambda<1$. If $\lambda\geq1$, then $\Gamma$ lies entirely in $\theta\in(0,\pi)$ and for $s<0$, $\gamma_-$ lies at the right-hand side of $\Gamma$ and outside the region bounded by $\Gamma$. If $\lambda<1$ then this time $\Gamma$ has also points in the region $\theta\in(\pi,3\pi/2)$ and in this case $\gamma_-$ lies at the right-hand side of $\Gamma$ for $s<0$, but inside the region bounded by $\Gamma$. Then, $\gamma_-$ intersects $\Gamma$ and reaches the line $\theta=\pi$. In any case, after finite instants, both $\gamma_\pm$ have points lying in the region $\theta\in(0,\pi)$. Since $\gamma_+$ (resp. $\gamma_-$) cannot stay in the region $\theta\in(0,\pi/2)$ (resp. in $\theta\in(\pi/2,\pi)$), $\gamma_+$ (resp. $\gamma_-$) must converge to the line $\theta=\pi/2$. Furthermore, $\gamma_-$ must actually intersect $\theta=\pi/2$ at some $(x_-,\pi/2)$, while $\gamma_+$ can either converge to $e_0$ as $s\to\infty$, or also intersect $\theta=\pi/2$ at some $(x_+,\pi/2)$. See Fig. \ref{fig:fasesrot}, left for the case $\lambda\geq1$ and right for $\lambda<1$.

\begin{figure}[H]
\centering
\begin{tikzpicture}[scale=1]
\node[anchor=south east,inner sep=0] at (-1,0){\includegraphics[width=.4\textwidth]{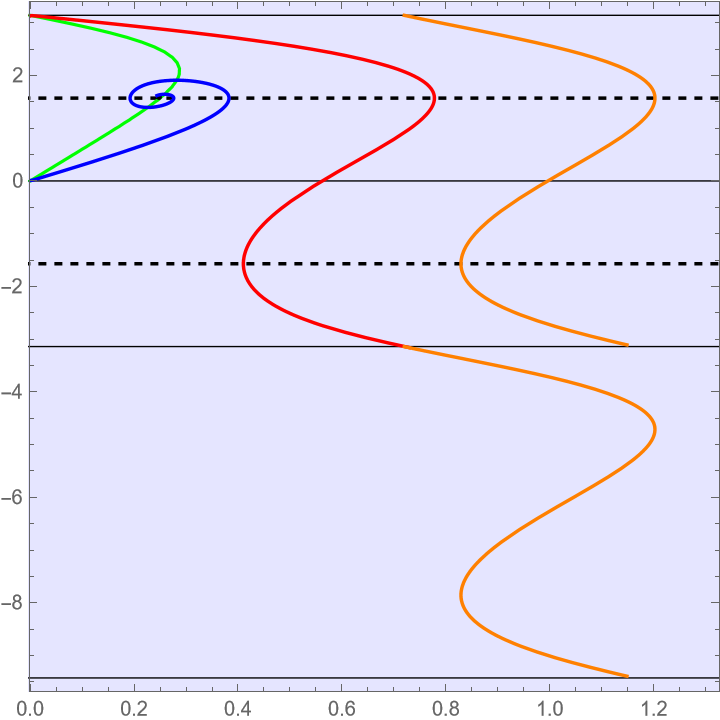}};
\node[anchor=south west,inner sep=0] at (1,0){\includegraphics[width=.4\textwidth]{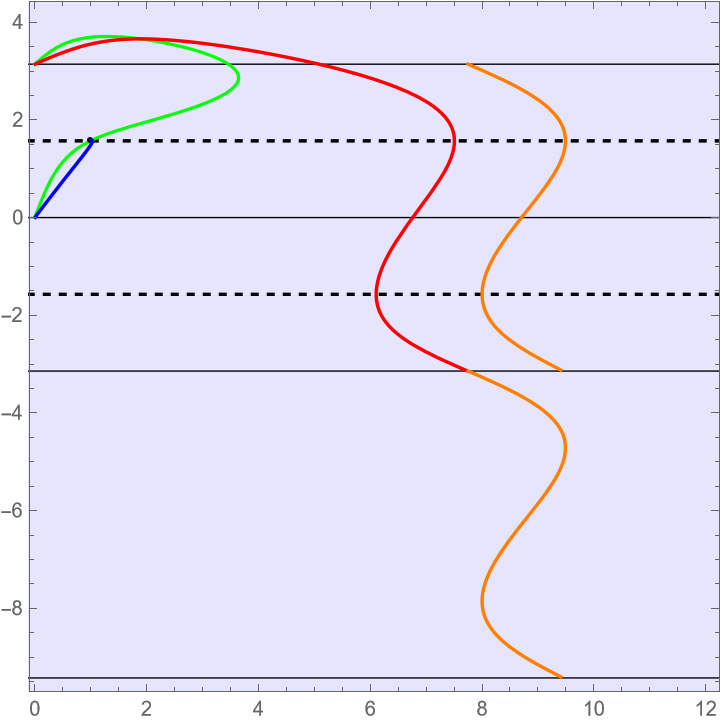}};
\end{tikzpicture}
\caption{The phase plane of Eq. \eqref{eqsystemrotational}. Left, for $\lambda\geq1$. Right, for $\lambda<1$. In both cases we have plotted the orbits $\gamma_+$ in blue and $\gamma_-$ in red. In orange we see the smooth prolongation of $\gamma_-$.}
\label{fig:fasesrot}
\end{figure}

Now we prove that if $\gamma_+$ intersects $\theta=\pi/2$ at $(x_+,\pi/2)$ with $x_+>1/(2\lambda)$, then $x_+<x_-$. If $x_+=x_-$, then we contradict Thm. \ref{thmnonexistenceclosed}, as the smooth orbit $\gamma_0=\gamma_+\cup\gamma_-$ would generate a topological sphere. If $x_+>x_-$, then $\gamma_-$ would lie at the left-hand side of $\gamma_+$ at the line $\theta=\pi/2$, and since $\gamma_-$ cannot intersect $\gamma_+$ the only possibility for $\gamma_-$ is to intersect $\Gamma$. Since $\gamma_-$ cannot converge to some $(0,\theta_0),\ \theta_0\in(0,\pi)$ by Prop. \ref{propintersectionorthogonal}, $\gamma_-$ must intersect again the line $\theta=\pi/2$, then $\Gamma$ and again $\theta=\pi/2$. In particular $\gamma(s)$ stays at a bounded distance to $e_0$ for $s\to-\infty$, which contradicts Cor. \ref{corcomportamientooribtasPB}. Therefore, $x_+<x_-$ and consequently $\gamma_+(s)\to e_0$ as $s\to\infty$. 

The profile curve $\alpha_+$ of the corresponding $\lambda$-translator $M_+$ is complete, has strictly increasing $z$-function and thus $M_+$ is embedded. Furthermore, $M_+$ converges to the vertical cylinder of constant mean curvature $\lambda$ and of radius $1/(2\lambda)$. See Fig. \ref{fig:rotlambdamayorigual1}, left.

\begin{figure}[H]
\centering
\begin{tikzpicture}[scale=1]
\node[anchor=south east,inner sep=0] at (-1.5,0){\includegraphics[width=.4\textwidth]{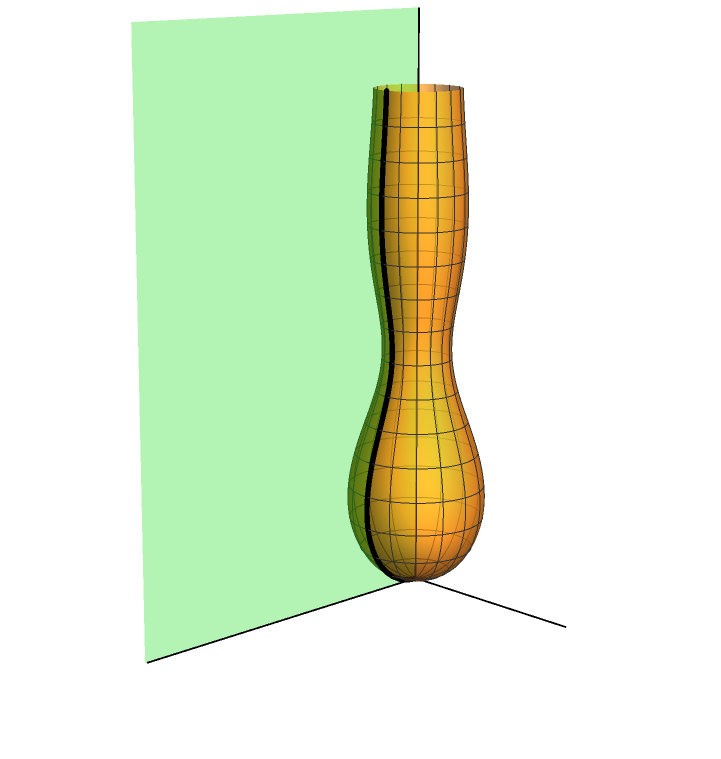}};
\node[anchor=south west,inner sep=0] at (1,0){\includegraphics[width=.5\textwidth]{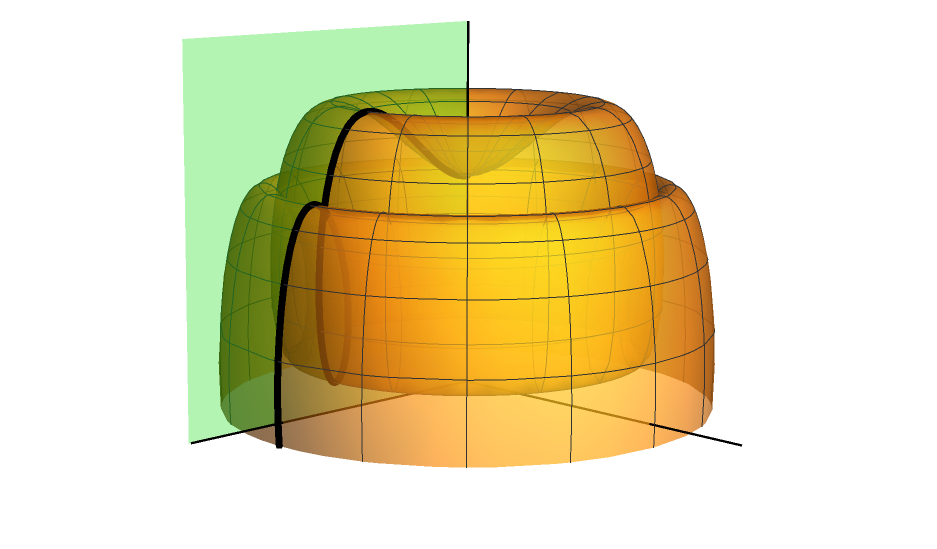}};
\node[anchor=south,inner sep=0] at (0,0){\includegraphics[width=.4\textwidth]{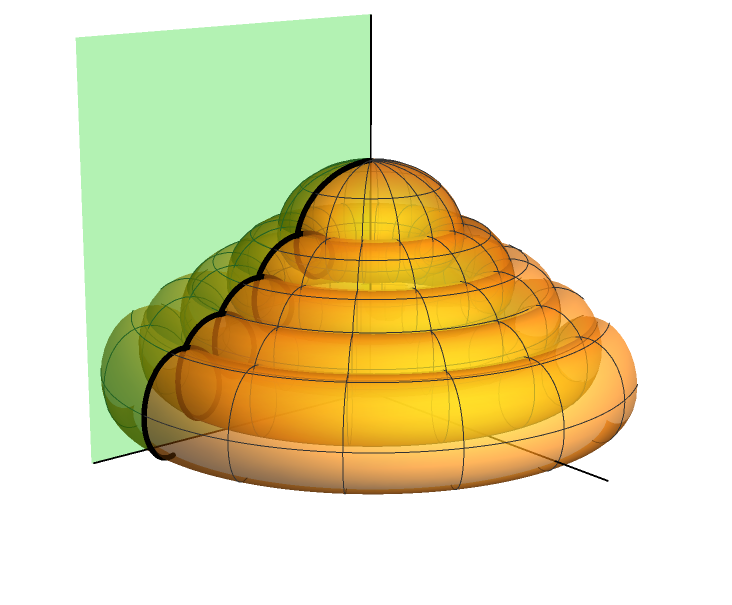}};
\end{tikzpicture}
\caption{Rotational $\lambda$-translators intersecting orthogonally the rotation axis. Left: the unit normal at the rotation axis is $E_3$. Center: $\lambda\geq1$ and the unit normal at the rotation axis is $-E_3$. Right: $\lambda<1$ and the unit normal at the rotation axis is $-E_3$.}
\label{fig:rotlambdamayorigual1}
\end{figure}

Regarding $\gamma_-$, it intersects the line $\theta=0$ at $(x_0,0)$ for some $x_0>0$, and then enters the region $\theta\in(-\pi,0)$ where the curve $\Gamma$ does not appear. Since $\gamma_-$ cannot converge to $\theta=0$, it reaches the line $\theta=-\pi/2$ where its $x$-coordinate reaches a local minimum and then increases, and finally $\gamma_-$ intersects $\theta=–\pi$ at some $(x_1,-\pi)$ with $x_1>0$. Now we take advantage of the $2\pi$-periodicity of $\R$ in the $\theta$-direction as in the previous section. We consider the orbit $\gamma_1$ such that $\gamma_1(0)=(x_1,\pi)$. Since $\gamma_1$ and $\gamma_-$ cannot intersect then $\gamma_1$ lies at the right-hand side of $\gamma_-$ in the region $\theta\in(-\pi,\pi)$ and eventually passes through some $(x_2,-\pi)$ with $x_2>x_1$. By $2\pi$-periodicity we can consider the translation $\gamma_1-(0,2\pi)$, which agrees with $\gamma_-$ at $(x_1,-\pi)$. By uniqueness, we can smoothly glue $\gamma_-$ and $\gamma_1$ to form a larger orbit, that will be still denoted by $\gamma_-$, that reaches the line $\theta=-3\pi$ exactly at $(x_2,-3\pi)$.

This process is repeated and we obtain a complete orbit $\gamma_-$ that has $(0,\pi)$ as endpoint. Although at this point is not crucial, we prove that the $x$-coordinate of $\gamma_-$ diverges. This will follow if we show that the orbits $\gamma_n$ generated in the region $\theta\in(0,\pi)$ for being glued with $\gamma_-$ do not accumulate to a finite orbit $\gamma_\infty$. If this was the case, let $(x_\infty^1,0)$ and $(x_\infty^2,\pi)$ denote the endpoints of $\gamma_\infty$. By Prop. \ref{propextremosorbita} we know that $x_\infty^1>x_\infty^2$. For an orbit $\gamma_n$ we denote as well by $(x_n^1,0)$ and $(x_n^2,\pi)$ its endpoints. Hence, $x_n^k\to x_\infty^k$ for $k=1,2$. In particular, there exists $n_0$ big enough such that $x_{n_0}^1>x_\infty^2$, and thus the next orbit $\gamma_{n_0+1}$ would lie at the right-hand side of $\gamma_\infty$, a contradiction.

The corresponding profile curve $\alpha_-$ is complete, both its $x$ and $z$-functions increase and decrease and in particular $\alpha_-$ has infinitely-many self-intersections. Finally, depending on if $\lambda\geq1$ or $\lambda<1$, the behavior of $\alpha_-$ at the rotation axis changes. Recall that at such intersection, the unit normal in any case is $-E_3$. If $\lambda\geq1$, then $\alpha_-$ intersects the rotation axis with strictly increasing $z$-function, hence around such intersection the surface is strictly convex in the sense that $M_-$ lies above the plane $\{z=z(0)\}$. See Fig. \ref{fig:rotlambdamayorigual1}, center. If $\lambda<1$, then $\alpha_-$ intersects the rotation axis with strictly decreasing $z$-function, hence around such intersection the surface is strictly concave in the sense that $M_-$ lies below the plane $\{z=z(0)\}$. See Fig. \ref{fig:rotlambdamayorigual1}, right.

\subsection{Proof of Thm. \ref{t2}}\label{s45}

Next, we classify the rotational $\lambda$-translators that do not intersect the rotation axis. The classification follows easily in virtue of all the study previously done, hence we just highlight the main features and skip the details.

Fix $x_0>0$ and let $\gamma$ be the orbit passing through $(x_0,\pi/2)$ at $s=0$. If $x_0=1/(2\lambda)$ we know that $\gamma$ is the equilibrium $e_0$ and the rotational $\lambda$-translator is the circular cylinder of CMC $\lambda$. As in the previous section, let $x_+<x_-$ denote the intersection of the orbits $\gamma_+,\gamma_-$, respectively, with $\theta=\pi/2$. If $x_0\neq 1/(2\lambda)$ then we distinguish cases. 

If $x_0<x_-$ then the orbit $\gamma_-$ acts as a barrier and necessarily $\gamma(s)\to e_0$ as $s\to\infty$; one can apply here Cor. \ref{corcomportamientooribtasPB}. The only point to remark here is that depending on if $\lambda\geq1$ (Fig. \ref{fig:rotlambdamayorigual11}, left) or $\lambda<1$ (Fig. \ref{fig:rotlambdamayorigual11}, right) the orbit $\gamma$ may intersect $\Gamma$ directly as $s\to\infty$ (the orbits in orange in Fig. \ref{fig:rotlambdamayorigual11}), or may intersect first the line $\theta=\pi$ before intersecting $\Gamma$ (the orbit in pink in Fig. \ref{fig:rotlambdamayorigual11}, right). When $s\to-\infty$ then $\gamma(s)$ stays at the left-hand side of $\gamma_-$ and since $\gamma_+$ acts as well as a barrier, the behavior of $\gamma$ is similar to the one of $\gamma_-$. See Fig. \ref{fig:rotlambdamayorigual11}, both orbits in orange and the one in pink, right.

\begin{figure}[H]
\centering
\begin{tikzpicture}[scale=1]
\node[anchor=south east,inner sep=0] at (-.5,0){\includegraphics[width=.4\textwidth]{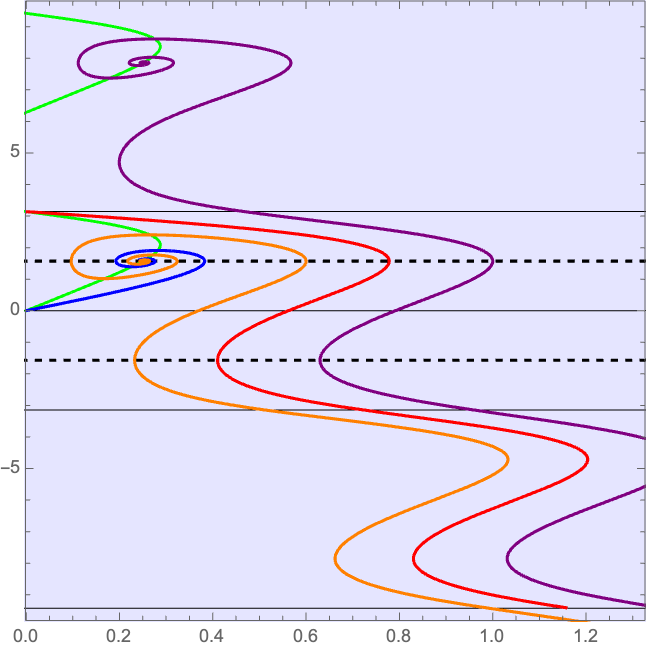}};
\node[anchor=south west,inner sep=0] at (.5,0){\includegraphics[width=.4\textwidth]{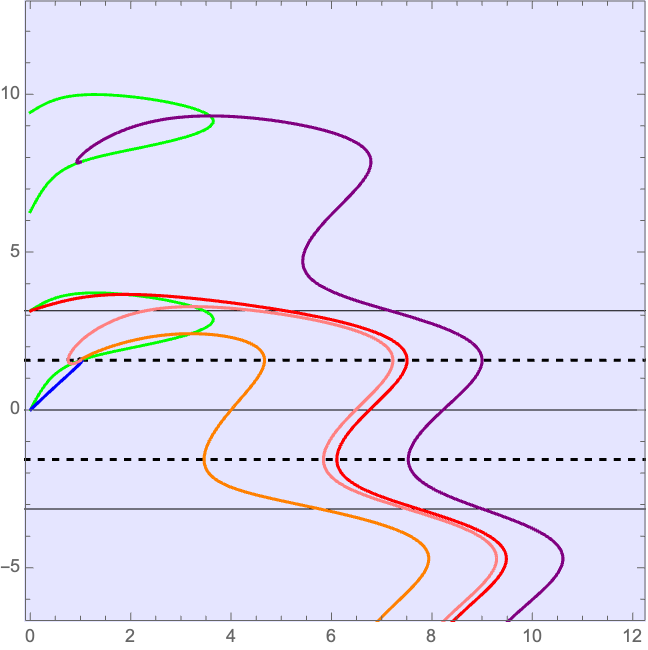}};
\end{tikzpicture}
\caption{The phase plane of Eq. \eqref{eqsystemrotational}. Left: for $\lambda\geq1$; right: for $\lambda<1$. In both phase planes there are orbits corresponding to rotational $\lambda$-translators not intersecting the rotation axis.}
\label{fig:rotlambdamayorigual11}
\end{figure}

If $x_0>x_-$ the behavior of $\gamma(s)$ as $s\to-\infty$ is clear, as $\gamma$ stays at the right-hand side of $\gamma_-$. If $s\to\infty$ then the $x$-coordinate of $\gamma$ tends to decrease as $\gamma$ eventually converges to some equilibrium $e_0+(0,2k\pi)$ for some $k>0$. By $2\pi$-periodicty in the $\theta$-direction of $\R$, this orbit could be realized as one of the previously studied item by just considering the first value $\hat{x_0}>0$ such that $\gamma$ passes through $(\hat{x_0},\pi/2+2k\pi)$ and $\hat{x_0}<x_-$. See Fig. \ref{fig:rotlambdamayorigual11}, both orbits in purple. We deduce due to the $2\pi$-periodicity that the fact that $x_0<x_-$ or not is irrelevant, as any orbit will eventually converge to some equilibrium as $s\to\infty$.

The corresponding rotational $\lambda$-translator is a properly immersed annulus, with one end being embedded and converging to the CMC cylinder of radius $1/(2\lambda)$ and the other with unbounded distance to the rotation axis and self-intersecting infinitely many times. See Fig. \ref{fig:rotlambdamayorigual11}, left for the cases $\lambda\geq1$; center for $\lambda<1$ and the orbit does not intersect $\theta=\pi$; and right for $\lambda<1$ and the orbit intersects $\theta=\pi$.
 
\begin{figure}[H]
\centering
\begin{tikzpicture}[scale=1]
\node[anchor=south east,inner sep=0] at (-1.5,0){\includegraphics[width=.4\textwidth]{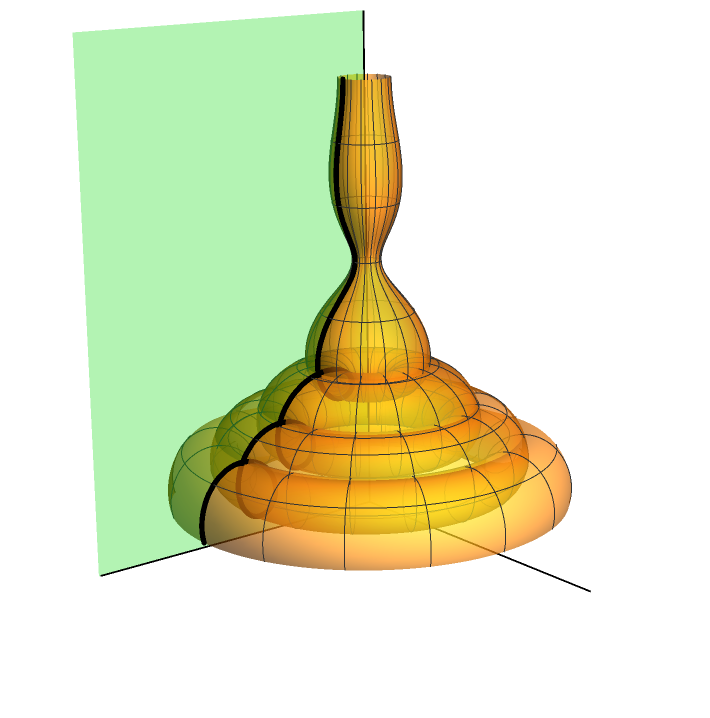}};
\node[anchor=south,inner sep=0] at (.4,0){\includegraphics[width=.4\textwidth]{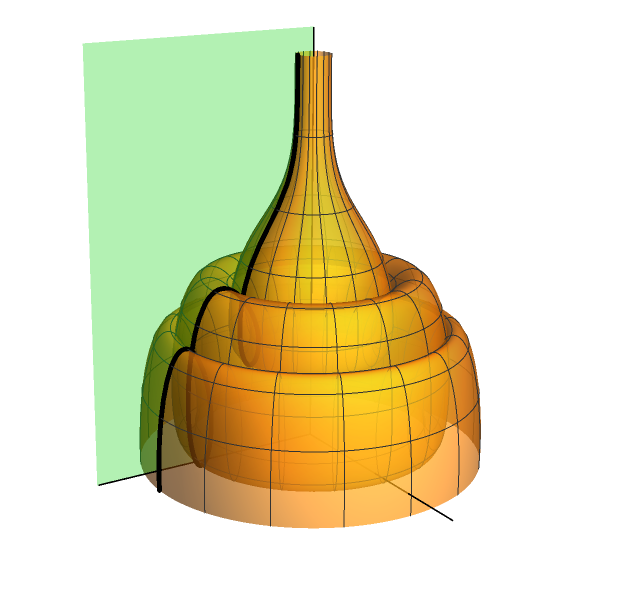}};
\node[anchor=south west,inner sep=0] at (2,0){\includegraphics[width=.4\textwidth]{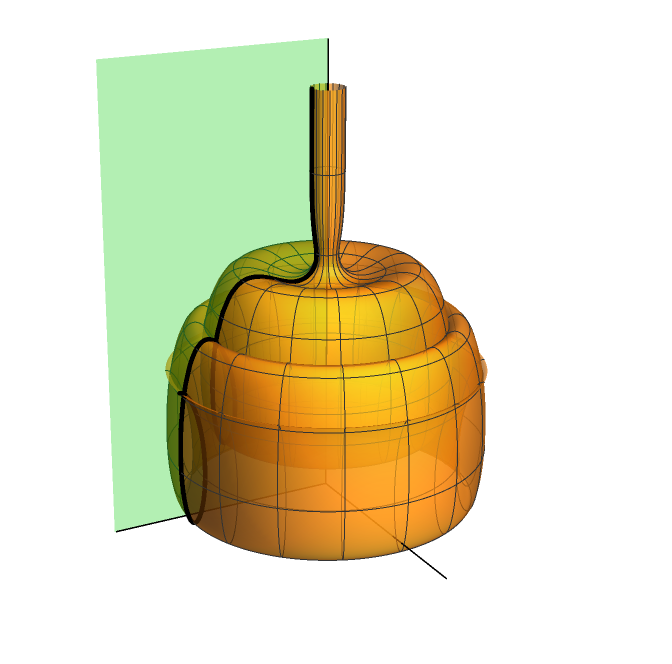}};
\end{tikzpicture}
\caption{Rotational $\lambda$-translators not intersecting the rotation axis. Left: $\lambda\geq1$; center: $\lambda<1$ and the corresponding orbit does not intersect $\theta=\pi$; right: $\lambda<1$ and the corresponding orbit intersects $\theta=\pi$.}
\label{fig:rotlambdamayorigual11}
\end{figure}

\section{Helicoidal $\lambda$-translators}\label{s5}

In the final section of this paper, we classify helicoidal $\lambda$-translators. An helicoidal motion in $\nil$ of \emph{pitch} $h>0$ is the one-parameter group of isometries defined as the composition of a rotation $\rho_t$ and a vertical translation $V_{ht}$,
$$
\varpi_t=\rho_t\circ V_t (x,y,z)=(x\cos t+y\sin t,x\sin t-y\cos t,z+ht).
$$
As in the rotational case, a helicoidal surface is defined as the image of a planar curve $\alpha(s)=(x(s),0,z(s))$, called the base or generating curve, by the one-parameter group of helicoidal motions $\varpi_t$. Thus, a parametrization is
$$
\psi(s,t)=(x(s)\cos t,x(s)\sin t,z(s)+ht).
$$
Although the intermediate computations will be omitted, we highlights the main features. The first fundamental form has determinant
$$
\langle\psi_s,\psi_s\rangle\langle\psi_t,\psi_t\rangle-\langle\psi_s,\psi_t\rangle^2=((h-\tau x^2)^2+x^2)x'+x^2z'^2,
$$
hence we consider the arc-length parameter
$$
x'=\frac{\cos\theta}{\sqrt{(h-\tau x^2)^2+x^2}},\qquad z'=\frac{\sin\theta}{x}.
$$
With this arc-lenght parameter, the mean curvature and angle function are
$$
2H=\frac{\sin\theta}{x}+\theta'\sqrt{(h-\tau x^2)^2+x^2},\qquad \langle N,E_3\rangle=\frac{x\cos\theta}{\sqrt{(h-\tau x^2)^2+x^2}}.
$$
Consequently, the following system is fulfilled by the coordinate curves of a helicoidal $\lambda$-translator.

$$
\left\lbrace
\begin{array}{l}
\vspace{.2cm} x'=\dfrac{\cos\theta}{\sqrt{(h-\tau x^2)^2+x^2}},\\
z'=\dfrac{\sin\theta}{x},\\
\theta'=\dfrac{1}{\sqrt{(h-\tau x^2)^2+x^2}}\left(2\left(\dfrac{\cos\theta}{\sqrt{(h-\tau x^2)^2+x^2}}+\lambda\right)-\dfrac{\sin\theta}{x}\right).
\end{array}
\right.
$$

\subsection{The phase plane of helicoidal $\lambda$-translators}\label{s51}
As in the previous section, we restrict to the $2D$-system
\begin{equation}\label{eqsystemhelicoidal}
\left\lbrace
\begin{array}{l}
\vspace{.2cm} x'=\dfrac{\cos\theta}{\sqrt{(h-\tau x^2)^2+x^2}},\\
\theta'=\dfrac{1}{\sqrt{(h-\tau x^2)^2+x^2}}\left(2\left(\dfrac{\cos\theta}{\sqrt{(h-\tau x^2)^2+x^2}}+\lambda\right)-\dfrac{\sin\theta}{x}\right),
\end{array}
\right.
\end{equation}
and define the phase plane $\R$ as the solutions (orbits) of Eq. \eqref{eqsystemhelicoidal}. We study its properties and its monotonicity regions by means of the implicit curve
$$
\Gamma=\{(x,\theta)\in\R\colon F(x,\theta)=0\},
$$
where
$$
F(x,\theta)=2x\left(\lambda+\dfrac{x\cos\theta}{\sqrt{(h-\tau x^2)^2+x^2}}\right)-\sin\theta.
$$
We see again that $e_0=(\frac{1}{2\lambda},\frac{\pi}{2})$ is the unique equilibrium of \eqref{eqsystemhelicoidal}, whose corresponding helicoidal surface is just the vertical CMC cylinder of radius $1/(2\lambda)$. 

Although the existence of the parameter $h$ complicates in substantial ways the study of the global properties of $\Gamma$, we are still able to prove its main properties. First, recall that for fixed datum $\tau,\lambda,h$, the equation $F(x,\theta)=0$ does not admit as solution a sequence $(x_n,\theta_n)\in\Gamma$ with $x_n\to\infty$, hence there exists $C=C(\tau,\lambda,h)>0$ such that $\Gamma\subset\{0\leq x\leq C\}$. Since $F(x,\theta)=0$ has no solution for $\theta\in(3\pi/2,2\pi)$, we conclude that $\Gamma$ is contained in the rectangle $[0,C]\times[0,2\pi]$ and since $\Gamma$ is closed we deduce that $\Gamma$ is compact, with finitely-many connected components. Furthermore, if $\lambda>1$ then the constant $C$ can be chosen to be independent of $h$.

Recall that $F(0,0)=F(0,\pi)=0$ and
$$
\frac{\partial F}{\partial x}(x,\theta)=2\left(\lambda+\frac{x(2h(h-\tau x^2)+x^2)\cos\theta}{((h-\tau x^2)^2+x^2)^{3/2}}\right),\qquad \frac{\partial F}{\partial \theta}(x,\theta)=-\cos\theta-\frac{2x^2\sin\theta}{\sqrt{(h-\tau x^2)^2+x^2}},
$$
hence
$$
\frac{\partial F}{\partial \theta}(0,0)=-1,\qquad \frac{\partial F}{\partial \theta}(0,\pi)=1,
$$
and $\Gamma$ is expressed locally as a vertical graph $\theta=\theta(x)$ around $(0,0)$ and $(0,\pi)$. Moreover, around $(0,0)$ (resp. around $(0,\pi)$) we have $\theta'(0)=2\lambda$ (resp. $\theta'(0)=-2\lambda$) and thus both components of $\Gamma$ lie in the region $\theta\in(0,\pi)$ locally around both $(0,0)$ and $(0,\pi)$, regardless of the value of $\lambda$. This is a first difference with the rotational case, since there for $\lambda<1$ the curve $\Gamma$ had points above the line $\theta=\pi$ around $(0,\pi)$.

We begin by proving that for each $\theta_0\in(0,\pi/2]$ there exists a unique $x_0$ such that $(x_0,\theta_0)\in\Gamma$. This implies in particular that $\Gamma$ can be expressed in such a region for $\theta$ as a horizontal graph $x=x(\theta)$. So, fix $\theta_0$ and consider the 1-dimensional function $F(x,\theta_0)$. Then, 
$$
F(0,\theta_0)=-\sin\theta_0<0,\quad \lim_{x\to\infty}F(x,\theta_0)=\infty,\quad \frac{\partial F}{\partial x}F(x,\theta_0)>0,
$$
hence the existence of a unique $x_0>0$ such that $F(x_0,\theta_0)=0$ is assured. In particular, we deduce that $\Gamma$ has a single connected component in the region $[0,\pi/2]$, which agrees with the graph of $\theta=\theta(x)$ passing through $(0,0)$ since $\partial F/\partial \theta>0$ in this region.

Now we focus on the region $\theta\in(\pi/2,\pi]$ and begin by studying the intersection of $\Gamma$ with the line $\theta=\pi$. This will be the first difference arising due to the existence of the parameter $h>0$. We have
$$
F(x,\pi)=2x\left(\lambda-\frac{x}{\sqrt{x^2+(h-\tau x^2)^2}}\right).
$$ 
Thus:
\begin{itemize}
\item If $\lambda>1$ and $x>0$, we find no solutions of $F(x,\pi)=0$, hence $\Gamma$ intersects $\theta=\pi$ only at $x=0$.
\item If $\lambda=1$ and $x>0$, then $F(x,\pi)=0$ has a double solution $x_1^\pm=\sqrt{h/\tau}$. Furthermore, $\frac{\partial F}{\partial\theta}(x_1^\pm,\pi)=-1$ and $\frac{\partial F}{\partial x}(x_1^\pm,\pi)=0$ and hence $\Gamma$ is a graph $\theta=\theta(x)$ around $(x_1^\pm,\pi)$ that attains a global maximum at $x_1^\pm$. In particular, for each $h>0$ there exists $\theta_h$ close to $\pi$ such that $F(x,\theta_h)=0$ has exactly three solutions: one close to $(0,\pi)$ and the other two near $(x_1^\pm,\pi)$.
\item If $\lambda<1$ and $x>0$, then $F(x,\pi)=0$ has two well-defined real solutions
\begin{equation}\label{eqpuntosxmasmenos}
x_\lambda^{\pm}=\frac{\sqrt{1-\lambda^2+2h\tau\lambda^2\pm\sqrt{(1-\lambda^2)(1-\lambda^2+4h\tau\lambda^2)}}}{\lambda\tau\sqrt{2}}.
\end{equation}
We can also express $\Gamma$ as a vertical graph $\theta=\theta(x)$ around each of such points, and this time $\theta'(x_\lambda^-)>0$ and $\theta'(x_\lambda^+)<0$. We will deduce later that the component $\Gamma\cap\{\theta\in(\pi,3\pi/2)\}$ is a connected vertical graph that joins the points $(x_\lambda^-,\pi)$ and $(x_\lambda^+,\pi)$ and thus can be described by the function $\theta(x)$. Furthermore,
$$
\lim_{h\to\infty}x_\lambda^+-x_\lambda^-=\frac{\sqrt{1-\lambda^2}}{\lambda\tau}.
$$
Thus, even if $h\to\infty$, these points of intersection do not collapse to the same point.
\end{itemize}

We investigate the points of $\Gamma$ such that their $x$-coordinate has a local extrema. Around such points we can express $\Gamma$ as a horizontal graph $x=x(\theta)$ and $(x,\theta)\in\Gamma$ satisfies $x'(\theta)=0$ if and only if $\frac{\partial F}{\partial\theta}(x,\theta)=0$. Thus, points with local extrema at the $x$-coordinate are located at the intersection of $\Gamma$ with the graph of the function
$$
f(x)=\pi-\arctan\frac{\sqrt{(h-\tau x^2)^2+x^2}}{2x^2},\qquad x\geq0.
$$
We introduce the following notation: given a function $f$, we denote by $\mathcal{G}_f=\{(x,f(x))\colon x\in\mathrm{Dom}(f)\}$ to the graph of $f$. Note that $f$ extends smoothly to $x=0$ as $f(0)=\pi/2$ and $f'(0)=0$. Furthermore, $f'(x)=0$ for $x>0$ if and only if $x=\frac{\sqrt{2}h}{\sqrt{2h\tau-1}}$. This value is well-defined whenever $2h\tau>1$. Thus, if $2h\tau\leq1$ then $f$ is strictly increasing and otherwise has a local maximum whose value on $f$ is
$$
f\left(\frac{\sqrt{2}h}{\sqrt{2h\tau-1}}\right)=\pi-\arctan\frac{\sqrt{4h\tau-1}}{h}.
$$
In any case, $\lim_{x\to\infty}f(x)=\pi-\arctan\frac{\tau}{2}$ and in particular $\mathcal{G}_f$ lies contained in the region $\theta\in(\pi/2,\pi)$. By compactness of $\Gamma$, we ensure the existence of at least a point $(x_0,\theta_0)\in\Gamma$ such that $x'(\theta_0)=0$. Furthermore, if $\Gamma$ is connected then it intersects the $\gf$ at a unique point $(x_0,\theta_0)$ and $x_0$ must be a global maximum. Indeed, at any local extrema one has $x''(\theta)=-\frac{F_{\theta\theta}}{F_x}$ and since
$$
\frac{\partial^2F}{\partial\theta^2}(x,\theta)=-\frac{2x^2\cos\theta}{\sqrt{(h-\tau x^2)^2+x^2}}+\sin\theta>0,\quad \theta\in(\pi/2,\pi),
$$
we deduce that $x$ cannot have inflexion points, only local maxima and minima. Thus, if $\Gamma$ is connected it can only intersect $\gf$ at a unique point. We also deduce that for any connected component of $\Gamma\cap\{\theta\in[\pi,3\pi/2)\}$, the points with largest and lowest $x$-coordinate are precisely located at the line $\theta=\pi$. In particular, any such component can be expressed as a vertical graph $\theta=\theta(x)$. Let us study in more detail the existence of such points. The conditions $F(x,\theta)=0$ and $\frac{\partial F}{\partial\theta}(x,\theta)=0$ yield
$$
l(x):=\frac{1}{2x}+\frac{2x^3}{x^2+(h-\tau x^2)^2}=\frac{\lambda}{\sin\theta},
$$
and by solving $\theta$ we define $L(x)=\pi-\arcsin\frac{\lambda}{l(x)}$. Thus, the existence of points in $\Gamma$ with local critical values of the $x$-coordinate are located in the intersection of the graphs $\gf$ and $\gL$. See Fig. \ref{fig:grafosfyL}.

\begin{figure}[H]
\centering
\begin{tikzpicture}[scale=1]
\node[anchor=south east,inner sep=0] at (-1,0){\includegraphics[width=.35\textwidth]{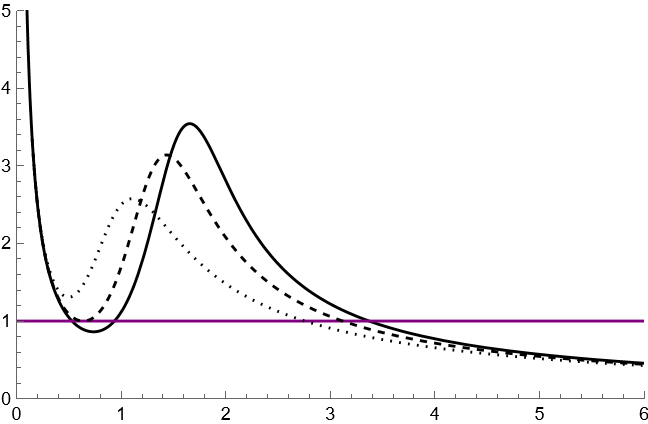}};
\node[anchor=south west,inner sep=0] at (1,0){\includegraphics[width=.45\textwidth]{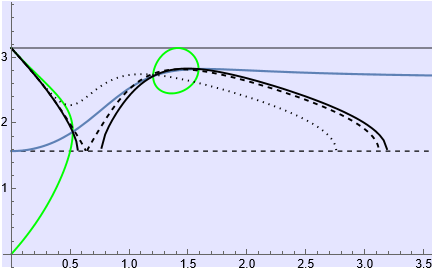}};
\end{tikzpicture}
\caption{Left: the graph of the function $l(x)$ for different values of the parameter $h$, and a horizontal line $y=\lambda$ for $\lambda=1$. Right, the graphs $\gf$ and $\gL$ in the phase plane $\R$. Every time that they intersect, the curve $\Gamma$ has a point with critical $x$-coordinate.}
\label{fig:grafosfyL}
\end{figure}

We study the properties of $L$ by means of the ones of $l$. First of all, recall that $L$ is only defined whenever $\lambda\leq l(x)$. Note that $\lim_{x\to0}l(x)=\infty$ and $L$ can be extended to $x=0$ as $L(0)=\pi$, and $\lim_{x\to\infty}l(x)=0$. Thus, for fixed $\lambda$ and $x$ sufficiently large $L(x)$ fails to be defined. Furthermore, there exists at least a value $x_0>0$ such that $l(x_0)=\lambda$ and therefore $\gL$ passes through the point $(x_0,\pi/2)$. Since $L(0)>f(0)$, $\gf$ and $\gL$ intersect at least once and we deduce again the existence of at least a point of $\Gamma$ with critical $x$-coordinate. Lastly, $L'(x)=0$ if and only if $l'(x)=0$ and if $x_0$ is such that $L'(x_0)=0$, then $L$ has a local maximum (resp. a local minimum) if and only if $l$ has at $x_0$ a local maximum (resp. a local minimum).

In Fig. \ref{fig:grafosfyL}, left, we can see the behavior of $\gl$ (plotted in black) for different values of the parameter $h$ and the line $y=\lambda$ for $\lambda=1$. In Fig. \ref{fig:grafosfyL}, right, we can see $\gL$ for the corresponding values of $h$ and also $\gf$. Each time that $l(x)$ intersects a horizontal line $y=\lambda$ at some $(x_0,l(x_0))$, we find that $\gL$ in $\R$ intersects the line $\theta=\pi/2$ at the point $(x_0,\pi/2)$, and the values of $x$ for which $l(x)<\lambda$ are not in the domain of $L$. In particular, $\gL$ has at most two connected components. We distinguish depending on the values of $\lambda$, starting from the limit case $\lambda=0$ and increasing it.
\begin{itemize}
\item If $ \lambda>0$ is small enough such that $l(x)=\lambda$ has only a solution $x_0$, then $\gL$ is a connected arc joining the point $(0,\pi)$ and $(x_0,\pi)$ and that intersects $\gf$ precisely at the unique point of $\Gamma$ with maximum $x$-coordinate.
\item  If $l(x)=\lambda$ has two solutions $x_0<x_1$, then the line $y=\lambda$ is tangent to $\gl$ at its local minimum point and thus $\gL$ is a connected arc that joins the points $(0,\pi)$ and $(x_0,\pi/2)$, then \emph{bounces} and joins the points $(x_0,\pi/2)$ and $(x_1,\pi/2)$.
\item If $l(x)=\lambda$ has three solutions $x_0<x_1<x_2$, then $\gL$ has two connected components: one, $\gL^1$, joining the points $(0,\pi)$ and $(x_0,\pi/2)$ and the other, $\gL^2$, joining $(x_1,\pi/2)$ and $(x_2,\pi/2)$.
\item When $\lambda$ further increases, the values $x_1$ and $x_2$ both converge to the value $x_+$ where $l$ reaches its local maximum. The component $\gL^2$ tends to be smaller, then collapses at the point $(x_+,\pi/2)$ and finally disappears.
\end{itemize}

Fix $\lambda>0$ and assume for a moment $h=0$, i.e. we are in the case of rotational surfaces. The monotony properties of $l$ depend on $\tau$. For instance, $l'(x)=0$ if and only if $\tau\leq 1/\sqrt{2}$ at the points
$$
x_\pm=\frac{1}{\sqrt{2-\tau^2\pm\sqrt{4-8\tau^2}}}.
$$
Hence, for $\tau\geq 1/\sqrt{2}$ the function $l$ is strictly decreasing, and if $\tau<1/\sqrt{2}$ then $l$ has a local minimum at $x_-$ and a local maximum at $x_+$. Now, for $h>0$ close enough to zero we have that $l$ has the same properties as for $h=0$ due to the continuity dependence of $h$. Let us denote by $x_\pm^h$ to the values where $l$ reaches its maximum and minimum, for given $h>0$. We find that as $h$ increases then $l(x_-^h)$ decreases and $l(x_+^h)$ increases; furthermore, we also see that $x_-^h$ and $x_+^h$ both increase and thus for fixed $\lambda>0$ the connected component $\gL^2$ always exists and has unbounded $x$-coordinate as $h$ increases.  

Now, for fixed $\tau$ and $\lambda$ and making $h$ vary, the behavior of $f$ and $L$ determines two different scenarios. In the first one, the graphs $\gf$ and $\gL$ intersect only once for every $h>0$. Thus, $\Gamma$ is connected and we are in an analog situation to rotational surfaces. In the second, the graphs $\gf$ and $\gL$ intersect only once for $h>0$ small enough; this always happens since for $h=0$ we only have one critical $x$-coordinate of $\Gamma$ (Fig. \ref{fig:comportamientoGammahelicoidales1}, left). As $h$ increases the local minimum of $\gL$ decreases and then intersects $\gf$ again, hence $\Gamma$ shrinks at such intersection (Fig. \ref{fig:comportamientoGammahelicoidales1}, center). For $h$ slightly greater we find two compact components of $\Gamma$: one, $\Gamma^1$, containing the points $(0,0)$ and $(0,\pi)$, and other, $\Gamma^2$ (Fig. \ref{fig:comportamientoGammahelicoidales1}, right). 

\begin{figure}[h]
\hspace{-1cm}
\begin{tikzpicture}[scale=.8]
\node[anchor=south east,inner sep=0] at (-3.5,0){\includegraphics[width=.4\textwidth]{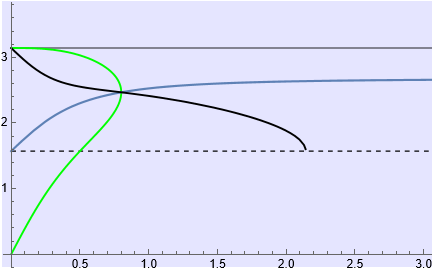}};
\node[anchor=south,inner sep=0] at (0,0){\includegraphics[width=.4\textwidth]{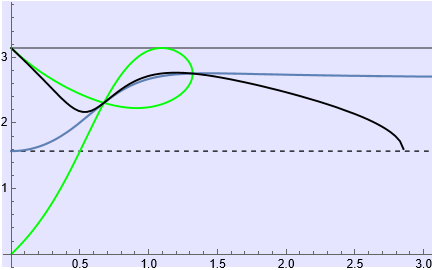}};
\node[anchor=south west,inner sep=0] at (3.5,0){\includegraphics[width=.4\textwidth]{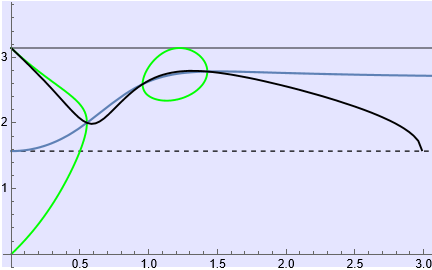}};
\end{tikzpicture}
\caption{The behavior of the curve $\Gamma$ and the graphs $\gf$ and $\gL$ for fixed $\tau,\lambda$ and varying $h$.}
\label{fig:comportamientoGammahelicoidales1}
\end{figure}

When $h$ further increases the component $\gL^2$ has greater $x$-coordinate and thus if $\lambda>1$ there exists $h_0>0$ such that the graph $\gL$ intersects $\gf$ tangentially one last time, making $\Gamma^2$ shrink and then disappears (Fig. \ref{fig:comportamientoGammahelicoidales2}, left).

If $\lambda=1$, there is always the point $(\sqrt{h/\tau},\pi)$ and $\Gamma^2$ is always locally a graph $\theta=\theta(x)$ around such point. However, as $h\to\infty$ the component $\Gamma^2$ tends to disappear. Indeed, assume that there exists $\theta_0\in(\pi/2,\pi)$ such that for every $h_n>0$ there exists $x_n>0$ big enough such that $(x_n,\theta_0)\in\Gamma^2$. Hence,
$$
2x_n\left(1+\frac{x_n\cos\theta_0}{\sqrt{x_n^2+(h_n-\tau x_n^2)^2}}\right)=\sin\theta_0.
$$ 
The right-hand side is fixed; it can be as small as we want, but is fixed. If $h_n\to\infty$ and so $x_n\to\infty$, we do not mind if the factor $(h_n-\tau x_n^2)^2$ converges to some non-negative constant, even zero, or diverge, as the left-hand side eventually diverges, a contradiction; the fact that $\theta_0$ is fixed is crucial. Thus, $\Gamma^2$ tends to shrink but never disappears for a finite $h>0$ (Fig. \ref{fig:comportamientoGammahelicoidales2}, center).

If $\lambda<1$ we know that there always exist the points $x_\lambda^\pm$ of intersection between $\Gamma^2$ and $\theta=\pi$, given by Eq. \eqref{eqpuntosxmasmenos}. This connected component will be unbounded if we make $h\to\infty$. Since the difference of the points $x_\lambda^+-x_\lambda^-$ is bounded as $h\to\infty$, this component $\Gamma^2$ never disappears (Fig. \ref{fig:comportamientoGammahelicoidales2}, right).

\begin{figure}[h]
\hspace{-1cm}
\begin{tikzpicture}[scale=.8]
\node[anchor=south east,inner sep=0] at (-3.5,0){\includegraphics[width=.4\textwidth]{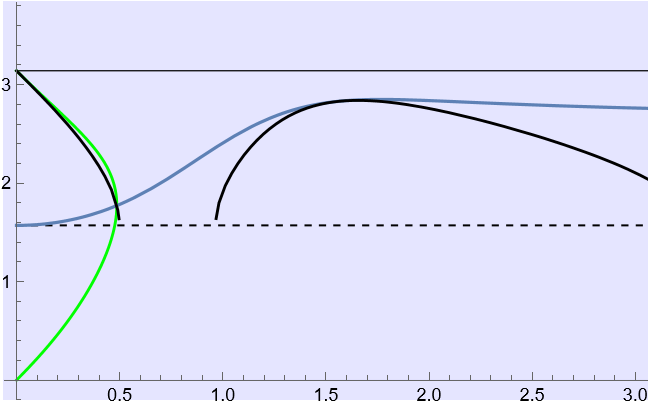}};
\node[anchor=south ,inner sep=0] at (0,0){\includegraphics[width=.4\textwidth]{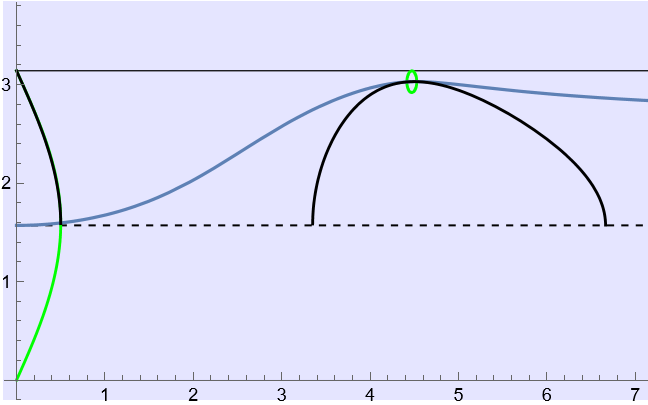}};
\node[anchor=south west,inner sep=0] at (3.5,0){\includegraphics[width=.4\textwidth]{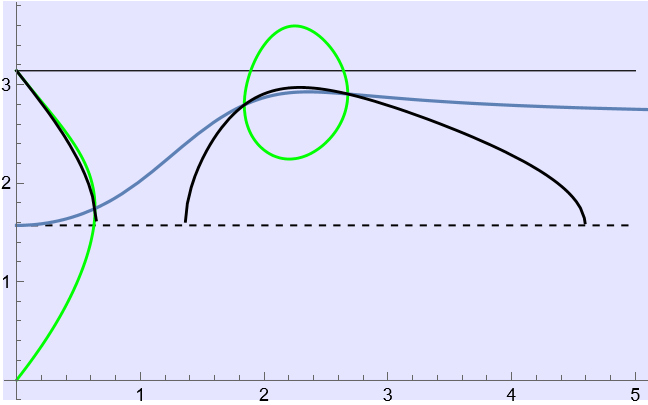}};
\end{tikzpicture}
\caption{The possibilities for the second connected component $\Gamma^2$ as $h$ increases. Left: it disappears ($\lambda>1$); center: it becomes smaller but does not disappear ($\lambda=1$); right: it always exist with a width bounded from below ($\lambda<1$).}
\label{fig:comportamientoGammahelicoidales2}
\end{figure}

\subsection{Proof of Thm. \ref{t3}}\label{s52}

The proof is immediate from the properties deduced in Section \ref{s41} about the phase plane and the properties of the curve $\Gamma$ proved in this section. The existence of a second connected component of $\Gamma$, which is a closed curve, only makes that an orbit entering its inner region changes the monotonicity of its $\theta$-coordinate, but the orbit eventually escapes from this region and then behaves as in Section \ref{s42}. The existence of the orbits $\gamma_\pm$ and their behaviors are proved the same as for the case $h=0$, thus details are skipped; see Fig. \ref{fig:phaseplanehelicoidal}, top, the orbits in blue and red, respectively. 

\begin{figure}[h]
\centering
\begin{tikzpicture}[scale=1]
\node[anchor=south,inner sep=0] at (0,.5){\includegraphics[width=.5\textwidth]{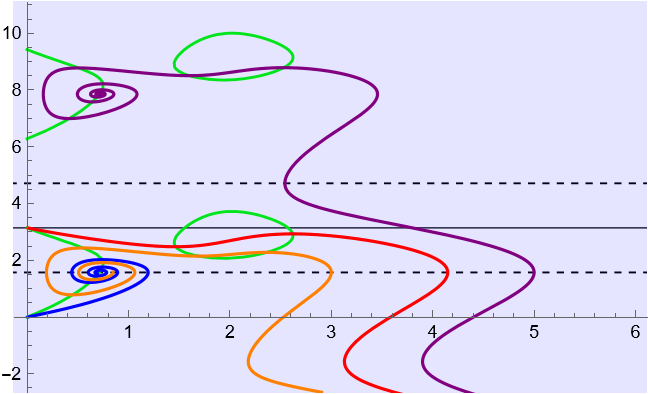}};
\node[anchor=north east,inner sep=0] at (-1.5,0){\includegraphics[width=.4\textwidth]{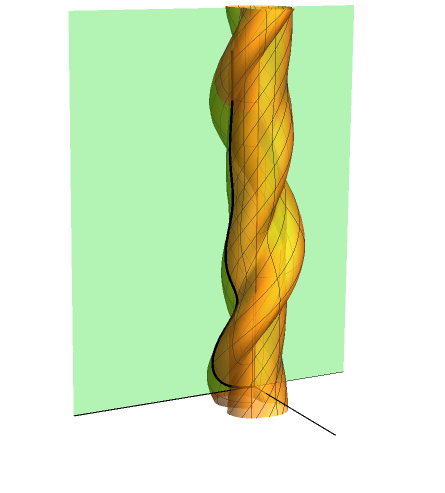}};
\node[anchor=north,inner sep=0] at (0,0){\includegraphics[width=.3\textwidth]{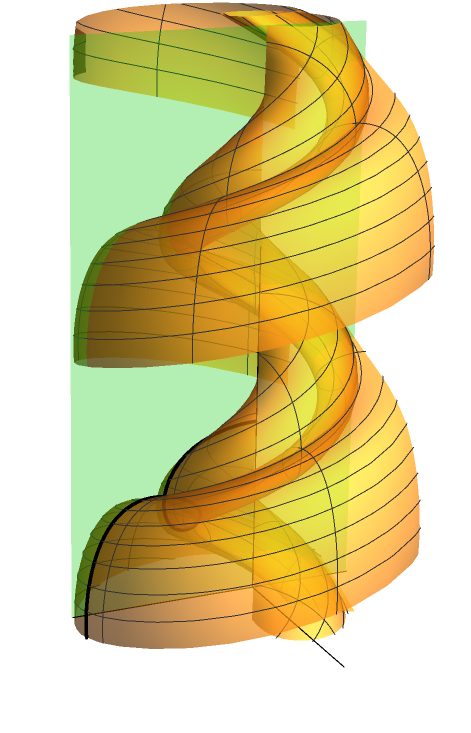}};
\node[anchor=north west,inner sep=0] at (2,0){\includegraphics[width=.35\textwidth]{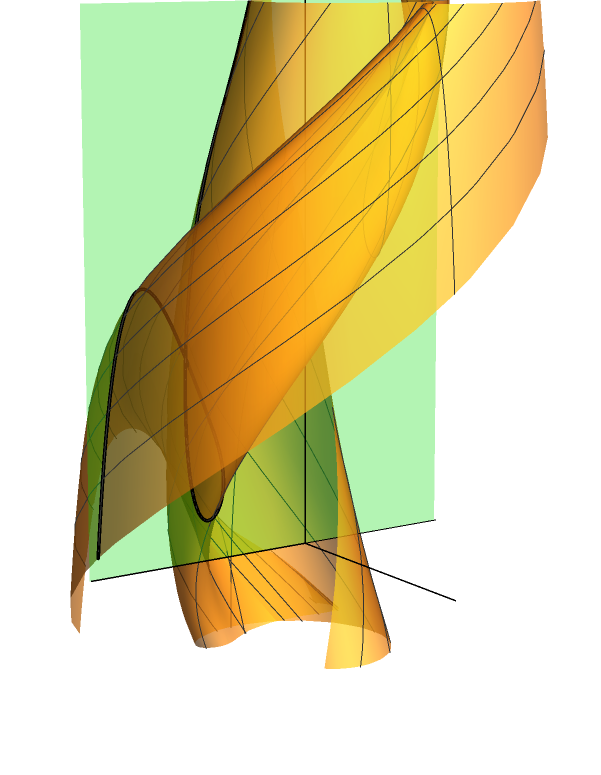}};
\end{tikzpicture}
\caption{The phase plane for helicoidal $\lambda$-translators. Here, $\lambda<1$.}
\label{fig:phaseplanehelicoidal}
\end{figure}


Let $(x_+,\pi/2)$ and $(x_-,\pi/2)$ with $x_+<x_-$ the points through which pass the orbits $\gamma_+$ and $\gamma_-$, respectively. To determine the rest of the orbits, fix some $x_0>0$ and let $\gamma(s)$ be the orbit passing through $(x_0,\pi/2)$ at $s=0$. If $x_0<x_-$, then $\gamma(s)\to e_0$ as $s\to\infty$, and when $s\to-\infty$, then $\gamma$ lies at the left-hand side of $\gamma_-$ and its $x$-coordinate diverges. See Fig. \ref{fig:phaseplanehelicoidal}, top, the orbit in orange. If $x_0>x_-$, then if $s\to-\infty$ then $\gamma$ lies at the right-hand side of $\gamma_-$ and its $x$-coordinate diverges as well, while if $s\to\infty$ then the $x$-coordinate of $\gamma$ decreases. Furthemore, let $(x_n,(2n-1)\pi)$ the intersection of $\gamma$ with $\theta=(2n-1)\pi$ and consider the translation $\gamma_n=\gamma-(0,2(n-1)\pi)$ restricted to $\theta\in[0,\pi]$; observe that $\gamma_1=\gamma$ in this region. A gluing argument as in Thm. \ref{t1} ensures that some $\gamma_{n_0}$ lies at the left-hand side of $\gamma_-$ and thus $\gamma_{n_0}(s)\to e_0$ as $s\to\infty$. Consequently, $\gamma(s)$ ends up converging to the equilibrium $e_0+(0,2(n_0-1)\pi)$ as $s\to\infty$; see Fig. \ref{fig:phaseplanehelicoidal}, top, the orbit in purple. Similarly as for rotational $\lambda$-translators not-intersecting the rotation axis and by $2\pi$-periodicity, the fact that $x_0<x_-$ or not is irrelevant as any orbit will eventually converge to some equilibrium as $s\to\infty$. 

In Fig. \ref{fig:phaseplanehelicoidal}, bottom, we see three helicoidal $\lambda$-translators. The first and the second correspond to the orbits with endpoints $(0,0)$ and $(0,\pi)$ plotted in blue and red, respectively. The third corresponds to an orbit converging to $e_0$ and not reaching the $\theta$-axis, hence the profile curve does not intersect the rotation axis.


\subsection*{Ethics declarations}

Conflict of interest. The author have no conflict of interest to declare that are relevant to the content of this article. No data were used to support this study

\subsection*{Acknowledgment}

Partially supported by the grant PID2021-124157NB-I00, funded by MCIN/ AEI/10.13039/501100011033/ "ERDF A way of making Europe".

\noindent
Universidad Murcia. \\ 
\emph{E-mail address:} jabueno@um.es

\end{document}